\documentclass[12pt,reqno]{amsart} 

\usepackage{amsmath,amsthm,amssymb,amsxtra}

\usepackage{charter}

\usepackage{amsfonts,amssymb}

\usepackage{microtype}

\usepackage{geometry} 
\usepackage{graphicx} % support the \includegraphics command and options

\usepackage{booktabs} % for much better looking tables
\usepackage{array} % for better arrays (eg matrices) in maths
\usepackage{paralist} % very flexible & customisable lists (eg. enumerate/itemize, etc.)
\usepackage{verbatim} % adds environment for commenting out blocks of text & for better verbatim
\usepackage{subfig} % make it possible to include more than one captioned figure/table in a single 

\usepackage[parfill]{parskip} % Activate to begin paragraphs with an empty line rather than an indent
\usepackage[toc,title,page]{appendix}
\appendixtitleoff
\appendixpageoff

\usepackage[english]{babel}
\usepackage{indentfirst,bm}
\usepackage{color}
\usepackage{calrsfs}
\usepackage[normalem]{ulem}
\usepackage[makeroom]{cancel}

\usepackage{enumitem}
\setlist[enumerate,1]{label=(\alph*), ref=\thethm.(\alph*)}

\usepackage{tikz,tikz-cd}
\usetikzlibrary{calc, plotmarks, arrows}

\usepackage{bookmark}
\usepackage{hyperref}
\hypersetup{
    bookmarks=true,         % show bookmarks bar?
    colorlinks=true,       % false: boxed links; true: colored links
    linkcolor=cyan,          % color of internal links
    citecolor=cyan,        % color of links to bibliography
    urlcolor=cyan           % color of external links
}

\newcommand{\case}[1]{\subsubsection{\bf #1}\mbox{}}

\definecolor{greenbean}{RGB}{199,237,204}

\definecolor{blue}{rgb}{0,0,1}
\def\blue{\color{blue}}
\definecolor{red}{rgb}{1,0,0}

\definecolor{green}{rgb}{0,1,0}

\definecolor{gray}{rgb}{.5,.5,.5}

\definecolor{yellow}{rgb}{1,1,.4}

\definecolor{purple}{rgb}{1,0,1}

\definecolor{gold}{rgb}{.5,.5,.2}

\theoremstyle{plain}
\newtheorem{thm}{Theorem}[section]
\newtheorem{lemma}[thm]{Lemma}
\newtheorem{prop}[thm]{Proposition}
\newtheorem{cor}[thm]{Corollary}
\newtheorem{conj}[thm]{Conjecture}

\newtheorem*{mainthm}{Main Theorem}

\theoremstyle{definition}
\newtheorem{Def}[thm]{Definition}

\theoremstyle{remark}
\newtheorem{rmk}[thm]{Remark}
\newtheorem{ex}[thm]{Example}

\newcommand{\mult}{\mathrm{mult}}
\newcommand{\ord}{\mathrm{ord}}
\newcommand{\vol}{\mathrm{Vol}}
\newcommand{\ceil}[1]{\lceil #1 \rceil}
\newcommand{\floor}[1]{\lfloor#1\rfloor}
\newcommand{\mc}[1]{\mathcal{#1}}
\newcommand{\mb}[1]{\mathbb{#1}}
\newcommand{\mf}[1]{\mathfrak{#1}}
\renewcommand{\d}{\mathrm{d}}
\renewcommand{\O}{\mc{O}}

\newcommand{\im}{\mathrm{Im}}

\newcommand{\df}{\mathrm{def}}
\newcommand{\mld}{\mathrm{mld}}
\newcommand{\supp}{\mathrm{Supp}}
\newcommand{\diff}{\mathrm{Diff}^*}

\newcommand{\lsyst}[1]{\lvert#1\rvert}
\newcommand{\Exc}{\mathrm{Exc}}
\newcommand{\length}{\mathrm{length}}
\newcommand{\m}{\mathfrak{m}}
\newcommand{\IN}{\mathrm{in}}
\newcommand{\spec}{\mathrm{Spec}}

\makeatletter
\newcommand{\hcell}[1]{\ifmeasuring@#1\else\omit$\displaystyle#1$\ignorespaces\fi}
\makeatother

\makeatletter
\newcommand*\rel@kern[1]{\kern#1\dimexpr\macc@kerna}
\newcommand*\widebar[1]{%
  \begingroup
  \def\mathaccent##1##2{%
    \rel@kern{0.8}%
    \overline{\rel@kern{-0.8}\macc@nucleus\rel@kern{0.2}}%
    \rel@kern{-0.2}%
  }%
  \macc@depth\@ne
  \let\math@bgroup\@empty \let\math@egroup\macc@set@skewchar
  \mathsurround\z@ \frozen@everymath{\mathgroup\macc@group\relax}%
  \macc@set@skewchar\relax
  \let\mathaccentV\macc@nested@a
  \macc@nested@a\relax111{#1}%
  \endgroup
}
\makeatother
\renewcommand{\bar}{\widebar}

\renewcommand{\tilde}{\widetilde}

\title{On Fujita's Freeness Conjecture in Dimension 5}

\author{Fei Ye}
\address{Department of Mathematics and Computer Science, Queensborough Community College - CUNY, 222-05 56th Ave. Bayside, NY 11364}
\email{feye@qcc.cuny.edu}

\author{Zhixian Zhu}
\address{KIAS,
85 Hoegiro, Dongdaemun-gu,
Seoul 130-722,
Republic of Korea}
\email{zhixian@kias.re.kr}

\begin{document}
\maketitle
\begin{abstract}
Let $X$ be a smooth projective variety of dimension $5$ and $L$ be an ample line bundle on $X$. We show that $|K_X+6L|$ is base-point free.
\end{abstract}

\section{Introduction}

Let $k$ be an algebraically closed filed of characteristic $0$. We always denote by $X$ a smooth projective variety over $k$. Let $L$ be an ample line bundle on $X$. The pluricanonical linear systems $\lsyst{mK_X}$ and more generally adjoint linear systems $\lsyst{K_X+kL}$ play very important roles in describing the structure of the projective variety $X$ and in classifying of projective varieties. An important question about linear systems is when they are base-point free.  In \cite{Fujita1988}, Fujita raised the following conjecture.

\begin{conj}[Fujita's freeness conjecture]
Let $X$ be a smooth projective variety and $L$ be an ample line bundle on $X$. Then the adjoint linear system $\lsyst{K_X+kL}$ is base-point free if $k\geq \dim X+1$. \end{conj}

Up to dimension 4, the conjecture has been proved (see \cite{Reider1988}, \cite{Ein1993a}, \cite{Kawamata1997}). In general, some effective results were obtained (see for example \cite{Demailly1993}, \cite{Kollar1993}, \cite{Angehrn1995},  \cite{Helmke1997}, \cite{Helmke1999} and \cite{Heier2002}). 

In higher dimensions, a successful approach is to use Kawamata-Veihweg vanishing theorem and run inductions on dimensions of minimal log canonical centers.

 Fujita \cite{Fujita1993} also introduced a new technique: volumes of graded linear systems on divisors. The technique was effectively used in \cite{Kawamata1997}, \cite{Helmke1999}, \cite{Lee1999}, \cite{Kakimi2000}, \cite{Ye2014} etc..

The idea is roughly the following. Instead of considering $kL$, we assume that $\sqrt[d]{L^d\cdot Z}\geq\sigma_d\geq  \dim X+1$ for any subvarieities $Z\subseteq X$. Given a point $x\in X$, by asymptotic Riemann-Roch theorem, we know that there exists an effective $\mathbb{Q}$--Cartier divisor $G$ linearly equivalent to $\lambda L$ with $0\leq \lambda<1$ such that $(X, G)$ is log canonical at $x$ and the minimal log canonical center $M(G)$ passing through $x$ is a normal subvariety. If $M(G)$ is a single point, then by Kawamata-Viehweg vanishing theorem we know that $\lsyst{K_X+L}$ is free at $x$. However, in general, $\dim M(G)$ may be positive. We have to run descending inductions on dimension of $M(G)$ to overcome this major difficulty. One way is to prove effective freeness on this singular variety $M(G)$ as in \cite{Ein1993a}. But it is extremely hard in higher dimensions due to possible bad singularities of $M(G)$.  Another approach\textemdash the idea is originally due to Angehrn and Siu and adapted by Helmke and many others\textemdash  is to create a new $\mb{Q}$--Cartier divisor $G'$ from $G$ so that $(X, G')$ is still log canonical at $x$ and the minimal log canonical center $M(G')$ is properly contained in $M(G)$.  

The difficulty of creating $(X, G')$ from $(X, G)$ with the expected properties is measured by the deficit $\df_x(X, G)$ (see Definition \ref{Def:deficit}) which was introduced by Ein \cite{Ein1997} and Helmke \cite{Helmke1997} independently.  They  showed that $\df_x(X,G)$ is closely related to the multiplicity $\mult_x M(G)$ and the order of vanishing $\ord_x G$, and is bounded from above by the dimension $\dim M(G)$.  For example, it is known that  $\df_x(G)\leq \dim X-\ord_x(G)$ and 
\begin{equation}\label{Hineq}
\mult_xM(G)\leq {e-\ceil{\df_x(X, G)}\choose e-d},
\end{equation}
where $e$ is the embedding dimension of $M(G)$ at $x$. (see Proposition \ref{prop:Ein-Helmke} for more relations).    In \cite{Helmke1997}, it was proved that we can create a desired new pair $(X, G')$ if the following inequality 
\[\left(\dfrac{\df_x(X, G)}{1-\lambda}\right)^d\cdot \mult_x M(G)< L^d\cdot M(G)\]
holds. To prove Fujita's freeness conjecture, it is very important to obtain optimal bounds for the three variables: $\mult_xM(G)$, $\df_x(X, G)$ and $\lambda$, where $\lambda$ is  the rational number such that $G$ is $\mb{Q}$--linearly equivalent to $\lambda L$ and $(X, G)$ is log canonical at $x$. 

In \cite{Kawamata1997} and \cite{Helmke1999}, it was observed that $\df_x(X, G)\leq 1$ if $x\in M(G)$ is a surface singularity.  In \cite{Ye2014}, we showed that in fact $\df_x(X, G)$ is bounded above by a function $\alpha_{\dim M(G)}(m)$ of $m=\mult_x(M(G))$ which is determined by the inequality \eqref{Hineq}. In particular, $\alpha_{2}(m)=1$. 

 Another effective way to bound  $\df_x(X, G)$  is to increase $\ord_xG$. When the minimal log canonical center is a divisor, Fujita \cite{Fujita1993} introduced a new technique\textemdash volume calculation of graded linear systems on Cartier divisors in smooth varieties\textemdash to argue that one can indeed find an effective $\mb{Q}$-Cartier divisor $G'$ that is $\mb{Q}$-linearly equivalent to $\lambda'L$ such that  $\lambda'<\lambda$ and $M(G')=M(G)$. The technique was effectively used in \cite{Kawamata1997}, \cite{Helmke1999}, \cite{Lee1999}, \cite{Kakimi2000}, \cite{Ye2014} etc.. In order to prove Fujita's freeness conjecture in dimension $5$, we generalize the volume calculation on Cartier divisors to Weil divisors (see Section \ref{sec:increasing-order} and Appedix \ref{sec:vol-Weil-div}) which is crucial to handle the most difficult case in our proof (see Section \ref{sec:main-proof} \ref{subsec:most-difficult}). 

Inspired by subadjunction theorem (see \cite[Theorem 2.5]{Ambro1999a}), we generalize deficit function to  lc  pairs and obtain the following inequality \[\df_x(X, G)\leq \df_x(M(G), \diff_{M(G)}(G)),\] where $\diff_{M(G)}(G)$ is the different of the pair $(X, G)$ on $M(G)$ as defined in \cite[Definition 4.2]{Ambro1999a}. Moreover, we show that  $\df_x(M(G), \diff_{M(G)}(G))$ is less than the minimal log discrepancy $\mld_x(M(G), \diff_{M(G)}(G))$. In particular, in the case that $\dim X=5$ and $\dim M(G)=2$, we have that \[\mld_x(M(G))\leq \frac{2}{\mult_xM(G)}.\] In fact, this inequality holds for any rational surface singularities (see Appendix \ref{sec:mld-def-surfaces}).

By studying invariants of singularities of minimal log canonical centers and carefully calculating volume of graded linear systems whose fixed parts contain a fixed prime divisor, we prove that Fujita's freeness conjecture in dimension $5$.

\begin{mainthm}
Let $X$ be a smooth projective variety of dimension $5$ and $L$ be an ample line bundle on $X$ such that $L^d\cdot Z\geq 6^d$ for any subvariety $Z\subseteq X$ of dimension $d$.  Then the adjoint linear system $\lsyst{K_X+L}$ is base-point free.
\end{mainthm}

This paper is organized as follows. In section \ref{Sec-deficit}, we generalize the deficit function to a klt pair and discuss its upper bounds. In section \ref{sec:singularity-of-center}, we discuss multiplicities and minimal log discrepancies at points in minimal log canonical centers. In section \ref{sec:increasing-order}, we show that the order of vanishing of an ample effective divisor can be increased when its minimal log canonical center is a divisor.  In section \ref{sec:cutting-center}, we present various criteria for cutting down minimal log canonical centers. We prove our main theorem in section \ref{sec:main-proof}. In Appendix \ref{sec:vol-Weil-div}, we estimate the volume of a linear system whose fixed part containing a prime divisor. Appendix \ref{sec:mld-def-surfaces} contains a proof by Jun Lu showing that the minimal log discrepancy of a rational surface singularity of multiplicity $m$ is at most $\frac{2}{m}$.\\

\noindent \textbf{Acknowledgements.}

The authors would like to thank Lawrence Ein, S\'andor Kov\'acs, Tommaso de Fernex, Shihoko Ishii, Mircea Musta\c{t}\v{a}, Florin Ambro, Linquan Ma and Yifei Chen for answering our questions and helpful discussions.  The first named author was partially supported by HKU Small Project Fund. He would like to thank Professor Ngaiming Mok for his advice and support. 

\section{Upper bounds for the deficit function\label{Sec-deficit}}
Due to unavoidable singularities of minimal log canonical centers, in this section, we generalize the definition of the deficit function to a log pair and study its upper bounds.  We show that the minimal log discrepancy of the minimal log center is an upper bound for the deficit function. 

Let $Y$ be a normal variety over $k$ and $x$ a closed point. Let $B$ be an effective $\mb{Q}$-Weil divisor on $Y$.  We call the pair $(Y, B)$ a log pair if  the divisor $K_Y+B$ is $\mathbb{Q}$--Cartier. 
Let $\mu: \tilde{Y}\rightarrow Y$ be a proper birational morphism with $\tilde{Y}$ normal.  On $\tilde{Y}$, there is a uniquely determined divisor $B^{\tilde{Y}}$ such that  $K_{\tilde{Y}}+ B^{\tilde{Y}}=\mu^*(K_Y+B)$ and  $B^{\tilde{Y}}=\mu^{-1}(B)$ on $\tilde{Y}\setminus \Exc(\mu)$, where $\Exc(\mu)$ is the exceptional locus. Write 
\[B^{\tilde{Y}}=\sum\limits_{E\subset \tilde{Y}} (1-a(E; Y, B))E,\] where $E$ is a prime divisor and $a(E; Y, B)$ is a  rational number, called  the {\em log discrepancy} of $E$ with respect to $(Y, B)$. The center $C_Y(E)$ of a prime divisor $E$ over $Y$ is defined as the closure of its image in $Y$.

A {\em log resolution} of a log pair $(Y, B)$ is a projective birational morphism $\mu: \tilde{Y}\to Y$ with $\tilde{Y}$ smooth such that  $\supp(\Exc(\mu)+ \mu^{-1}(B))$ is simple normal crossing.

We recall the definition of log canonical centers and the existence of minimal log canonical centers for a log canonical pair.

\begin{Def}
A pair $(Y, B)$ is said to be {\em log canonical} (resp. {\em Kawamata log terminal}\footnote{We will abbreviate this as klt.}) at  $x\in Y$ if it is a log pair and $a(E; Y, B)\geq 0$ (resp. $a(E; Y, B)>0$) for a log resolution $\mu: \tilde{Y}\to Y$ and any prime divisor $E$ on $\tilde{Y}$ such that $x\in C_Y(E)$. If $(Y, B)$ is log canonical (resp. klt) at every closed point in $Y$, then we call $(Y, B)$ a {\em log canonical} (reps. {\em klt}) pair.

An irreducible subvariety of $Y$ is called a {\em log canonical center} of $(Y,B)$ if it is the center of a prime divisor $E$ over $Y$ such that $a(E; Y, B)\leq 0$. 
\end{Def}

The existence of the minimal log canonical center of a log canonical but not Kawamata log terminal pair at a closed point  is well-known (see for example \cite{Ein1993a}, \cite{Helmke1997}, \cite{Kawamata1997}). 

\begin{thm}
Let $(Y, B)$ be a pair which is log canonical but not klt at $x$. Then there exists a minimal log canonical center $M(B)$ of $(Y, B)$ at $x$. Moreover, the minimal log canonical center $M(B)$ is normal.
\end{thm}

\begin{rmk}\label{rmk:critical}
In general, there might be not only one prime divisor $E$ on $\tilde{Y}$ such that $C_Y(E)=M(G)$. However, by a tie-breaking technique (see for example \cite{Kollar1997}, \cite{Kawamata1997}, \cite{Lee1999}), i.e. replacing $G$ by $G'=(1-\varepsilon_1)G+\varepsilon_2 D$, where $\varepsilon_1$ and $\varepsilon_2$ are sufficiently small positive numbers and $D$ is an effective ample $\mb{Q}$--Cartier divisor on $Y$, we may assume that  $(Y, G')$ is log canonical at $x$ and there is only one prime divisor $E$ on $\tilde{Y}$ such that $C_Y(E)=M(G')=M(G)$ at $x$. 

Practically, in this paper, both $G$ and $D$ will be $\mb{Q}$--linearly equivalent to an ample line bundle $L$ on  $X$, and as seen in \cite{Ye2014} a small perturbation of $G$ doesn't affect the inductions, so in the rest of the paper, we may assume that there is only one prime divisor $E$ such that $M(G)=C_{\blue X}(E)$. In this case, $G$ is said to be critical at $x$ and $M(G)$ is called a critical variety by Ein (see \cite[Definition 2.4]{Ein1997}) and is an exceptional locus by Ambro (see \cite[Section 1.4]{Ambro1999a}). 
\end{rmk}

Let $G$ be an effective $\mathbb{Q}$--divisor on $X$. The multiplier ideal sheaf of $G$ is defined as 
$$\mc{I}(G)=f_*\mathcal{O}_Y(K_{\tilde{X}}-\floor{f^*(K_X+G)}),$$ where $f: \tilde{X}\to X$ is a log resolution of $G$. We denote by $Z(G)$ the scheme defined by $\mc{I}(G)$. If $(X,G)$ is log canonical but not klt at $x$, then we denote by $M(G)$ the minimal log canonical center of $(X,G)$ at $x$.

As explained in the introduction, we want the minimal log canonical center $M(G)$ to be a point. In that case,  the following lemma, which is a consequence of Kawamata-Viehweg vanishing theorem,  shows that the adjoint linear system $\lsyst{K_X+L}$ is free at $x$.

\begin{lemma}\label{cor:0-dimb-p-f} Let $G$ be an effective $\mb{Q}$--divisor on $X$ such that $Z(G)$ (or more generally, $M(G)$) is $0$-dimensional at $x$. Assume that $A$ is an integral divisor on $X$ such that $A-(K_X+G)$ is nef and big. Then the linear system $\lsyst{A}$ is free at $x$. 
\end{lemma}

However, in general, $\dim M(G)$ may be positive, we want to modify the divisor $G$ to get a new divisor $G'$ such that $M(G')$ is properly contained in $M(G)$. The difficulty to create such a new divisor $G'$ is measured by a function called the deficit $\df_x(X, G)$ of $G$ at $x$.

\begin{Def}\label{Def:deficit}
Let $(Y, B)$ be a log pair and $G$ be an  effective $\mathbb{Q}$--Cartier divisor on $Y$ such that $(Y, B+G)$ is log canonical at  $x\in Y$.  

We define the {\it deficit} of $G$ at $x$ as
  \[\begin{split}\df_x(Y, B, G)
  :=\sup\{&\ord_xD\mid (Y, B+G+D) \text{~is log canonical at~} x,\\ &D \text{~is~} \textup{effective~ and~} \mathbb{Q}\textup{--Cartier~on~}Y\}.\end{split}\] 
For an effective $\mathbb{Q}$--Cartier divisor $D$ on a normal projective variety $Y$, we define the  {\it order of vanishing} of $D$ at $x$ as 
\[\ord_xD:=\max\limits_m\{\frac{\max \{q~|~ \mc{O}_Y(-mD)\subset \mf{m}_x^q\}}{m}~|~mD \textit{~is ~Cartier~at~} x\}.\] If $Y$ is smooth at $x$,  and $D$ is a Cartier divisor, then our definition of order vanishing agrees with the usual one.  

For a subvariety $Z\subset Y$ containing $x$, we define the {\it relative deficit} of $G$ over $Z$ at $x$ as \[\begin{split}\df_x^Z(Y,B, G):=\sup\{&\df_x(Y,B, G+D) | ~(Y,B+G+D)\textup{ is log canonical at } x,\\ &{M(B+G+D)=Z}, D \text{~is~} \textup{effective~} \mathbb{Q}\textup{--Cartier}\}.\end{split}\] If there is no effective $\mb{Q}$-Cartier divisor $D$ such that $(Y,B+G+D)$ is log canonical at $x$ with $M(B+G+D)=Z$, we define $\df_x^Z(Y,B,G)=0$. 
\end{Def}

We would like to remark that this concept was introduced on smooth varieties by Ein \cite{Ein1997} and Helmke \cite{Helmke1997, Helmke1999} independently. In \cite{Ambro1999a}, Ambro also defined a similar concept called building of singularities towards Fujita's freeness conjecture.

For simplicity, without causing confusion, we will drop ``$B$" or ``$G$" from $\df_x(Y, B, G)$ and $\df_x^Z(Y, B, G)$ if 
$B=0$ or $G=0$. In particular, for a smooth variety $X$, we will simply write $\df_x(X, G)$ and $\df_x^Z(X, G)$ for $\df_x(X, 0, G)$ and $\df_x^Z(X, 0, G)$ respectively.

From the definition, we see that $\df_x^Z(Y, B, G)=\df_x(Y, B,G)$ if $Z=M(G)$.  Moreover, we have $\df_x(X, 0)\leq \dim X$.

In the following,  we will present some upper bounds for the deficit function and the relative deficit function.  We first recall the definition of minimal log discrepancy.

\begin{Def} The {\it minimal log discrepancy} of a log pair $(Y, B)$ at a proper Grothendieck point $\eta\in Y$ is defined as
\[\mld_\eta(Y, B)=\inf\limits_{C_Y(E)=\bar{\eta}}a(E; Y, B),\]
where the infimum is taken among all prime divisors over $Y$ having $\bar{\eta}$ as the center. 
\end{Def}

Let $(Y, B)$ be a pair which is log canonical at a point $x\in Y$, and $Z$ be the minimal log canonical center of $(Y, B)$ passing through $x$.  On $Z$, there is a divisor $\diff_Z(B)$ called the different of $(Y, B)$ at $Z$ as defined in \cite[Definition 4.2]{Ambro1999a}. Here, the different is the discriminant $\mb{Q}$--divisor which contains no moduli part. In other words, it is the divisorial part of the different defined in \cite[Section 4]{Kollar2013}.
 
We recall the following result on $\diff_Z(B)$.
\begin{lemma}[{\cite[Lemma 4.2]{Ambro1999a}}]\label{Ambro-diff}
Let $(Y, B)$ be a pair that is log canonical  at  $x\in Y$ with the minimal log center $Z$ and $D$ be $\mb{Q}$--Cartier divisor on $Y$ such that $Z$ is not contained in the support of $D$. Assume that $(Y, B+D)$ is still log canonical at $x$.  Then
\[\diff_Z(B+D)=\diff_Z(B)+D|_Z.\]
\end{lemma}

In the following proposition, we summarize some upper bounds for deficit functions. 

\begin{prop}\label{prop:Ein-Helmke}~ Let $G$ be an effective $\mb{Q}$--divisor on a smooth projective variety $X$ and  $(Y, B)$ be a log pair.   Assume that $(X, G)$ and $(Y, B)$ are log canonical at $x$, where $x$ represents a closed point on $Y$ as well as on $X$. Denote by $M(G)$ the minimal log canonical center of $(X, G)$ at $x$.
  \begin{enumerate}
    \item\label{Ein-Helmke-def-up-first} $\df_x(X, G)\leq \df_x(X, 0)-\ord_x G\leq \dim X-\ord_xG.$
\item \label{Helmke-singular} Let $Z$ be a subvariety of $X$ containing $x$ and let $D$ be an effective $\mb{Q}$--divisor on $X$ such that $Z$ is not contained in the support of $D$ and $\df_x^Z(X, G+D)>0$. Then $\df_x^Z(X, G+D)\leq \df_x^Z(X, G)-\ord_xD$.
\item\label{hyperplane} If $\df_x(X, G)\geq 1$, then for any general hyperplane section $H$, the pair $(H, G|_H)$ is log canonical at $x$ with minimal log canonical center $M(G)\cap H$ and $$\df_x(X, G)=\df_x(H, G|_H)+1.$$
\item\label{Ein-Helmke-def-up-last} $\df_x(X,G)\leq\dim M(G)$.

\item \label{def-up-new} 
$\df_x(Y, B)\leq \mld_x(Y,B)$.
\item \label{def-pairs} $\df_x(Y, B)\leq \df_x(Z, \diff_Z(B))$, where $Z$ is the minimal log canonical center of $(Y, B)$ at $x$. 
 \end{enumerate}
\end{prop} 
\begin{proof}
The statements \ref{Ein-Helmke-def-up-first}-\ref{Ein-Helmke-def-up-last} can be found in \cite{Ein1997} and \cite{Helmke1999}. In fact,  \ref{Helmke-singular} is stated in the proof of Theorem 2.2 in Helmke. We now show that the statement \ref{def-up-new} holds.
For any effective $\mb{Q}$--Cartier divisor $D$ on $Y$ such that $\ord_xD>\mld_x(Y, B)$, we will show that $(Y, B+D)$ is not log canonical at $x$. Let $\varepsilon$ be a positive number such that $\varepsilon<\ord_xD-\mld_x(Y, B)$.
By the definition of minimal log discrepancy, we can choose a proper birational morphism
$\mu: \tilde{Y}\to Y$ such that $a(E; Y, B)\leq \mld_x(Y, B)+\varepsilon$ for a prime divisor $E$ over $x$.
Note that  $a(E; Y, B+D)=a(E; Y, B)-\ord_E\mu^*(D)$ and $\ord_E\mu^*(D)\geq \ord_xD$.
Then
\[a(E; Y, B+D)\leq a(E; Y, B)-\ord_xD< \mld_x( Y, B)+\varepsilon-(\mld_x(Y, B)+\varepsilon)=0.\] Therefore, $(Y, B+D)$ is not log canonical at $x$. Hence, by the definition of deficit, we see that $\df_x(Y, B)\leq \mld_x(Y, B)$.

To prove \ref{def-pairs}, we will show that $\ord_xD\leq \df_x(Z, \diff_Z(G))$ for any effective $\mb{Q}$--Cartier divisor $D$ such that $(Y, B+D)$ is log canonical at $x$. Assume that $(Y, B+D)$ is log canonical at $x$. Since $(Y, B)$ is log canonical at $x$ with the minimal log canonical center $Z=M(B)$, then $Z$ is not in the support of $D$ and $(Z, \diff_Z(B+D))$ is also log canonical by adjunction theorem (see \cite[Theorem 2.5]{Ambro1999a}). By Lemma \ref{Ambro-diff}, we know that $\diff_Z(B+D)=\diff_Z(B)+D|_Z$. Therefore, by the definition of deficit, we conclude that
\[\ord_xD\leq \ord_xD|_Z\leq \df_x(Z, \diff_Z(B))\]
which implies that $\df_x(Y, B)\leq \df_x(Z, \diff_Z(B))$.
\end{proof}

By Proposition \ref{Ein-Helmke-def-up-first} and the definition of relative deficit, we know that if the  minimal log canonical center $Z=M(G)$ of $(X, G)$ at $x$ is a hypersurface, then \[\df_x^Z(X, 0)=\df^Z_x(X, Z)=\df_x(X, Z)\leq \df_x(X, 0)-\ord_xZ.\]

\begin{lemma}[\cite{Ein1997}, \cite{Helmke1999}]\label{lemma:Ein-Siu}Let $G$ be an effective $\mathbb{Q}$--Cartier divisor on a klt log pair $(Y, B)$ such that $(Y, B+G)$ is log canonical at  $x\in Y$. Denote by $Z$ the minimal log canonical center of $(Y, B+G)$ at $x$. Let $D$ be an effective $\mathbb{Q}$--Cartier divisor on $Y$ such that $(Y, B+G+D)$ is still log canonical at $x$, $Z$ is not contained in $D$ and the minimal log canonical center $Z'$ of $(Y, B+G+D)$ is properly contained in $Z$. Then
\[\df_x(Y,B, G+D)\leq \df_x(Y,B, G)-\ord_xD|_Z.\]
\end{lemma}
For a smooth projective variety $X$, the result is stated in \cite[Lemma 4.6]{Ein1997}. In \cite{Helmke1999}, he stated this general lemma in the proof of Theorem 2.2.

\section{Singularities of the minimal log canonical center\label{sec:singularity-of-center}}
It is important to study singularities of $M(G)$, especially the multiplicity $\mult_xM(G)$. In \cite{Kawamata1997}, Kawamata initiated the study of subadjucntion formulae. Now we have the following characterization of singularities of minimal log canonical centers.

\begin{thm}[\cite{Fujino2012}]\label{thm:Kawamata1998}
 Assume that $(X,G)$ is log canonical at $x$ with the minimal log canonical center $Z$.  Then there exist an effective $\mb{Q}$--divisor $G_Z$ on $Z$ such that $$(K_X+G)|_Z\sim K_Z+G_Z$$ and the pair $(Z, G_Z)$ is klt at $x$. In particular, $Z$ has at most a rational singularity at $x$.
\end{thm}

\begin{rmk}\label{rmk:embedding-dim-surface-sing}
Assume that the minimal log canonical center $Z$ of $(X, G)$ is a surface. Since $Z$ has a rational singularity at $x$, then $\mult_xZ=e-1$, where $e$ is the embedding dimension of $Z$ at $x$.
\end{rmk}

We also note that an implicit upper bound for deficit is given by the inequality in the following theorem.
\begin{thm}[\cite{Helmke1997}]\label{Thm:boundofmult}
Assume that $(X,G)$ is log canonical at $x$ with the minimal log canonical center $Z$.  Let $e$ be the embedding dimension of $Z$ at $x$, $d=\dim Z$ and $m=\mult_xZ$. Then
  \[m\leq \begin{pmatrix} e -\ceil{\df_x(X, G)}\\ e-d\end{pmatrix}.\]
  In particular, \[m\leq \begin{pmatrix} e -1\\ d-1\end{pmatrix}.\] 
\end{thm}

\begin{rmk}If $H$ is a prime divisor and $(X, H)$ is log canonical at $x$ whose minimal log canonical center at $x$ is $H$ itself,  then $\df_x(X, H)=\dim X-\mult_xH$ which shows that the inequality in the above theorem is optimal (see Example 3.5 \cite{Helmke1999}).
\end{rmk}

\begin{Def}\label{Def:alpha}
 Keep the same assumption and notations as in Theorem \ref{Thm:boundofmult}, we define $\beta_{d,e}(m)\leq d$ to be largest solution of $y$ in the equation
  \[\begin{pmatrix} e -y\\ e-d\end{pmatrix}=m.\]  Denote by $\alpha_{d,e}(m)$ the largest integer less than or equal to $\beta_{d,e}(m)$.
\end{Def}

By the definition of $\alpha_{d, e}(m)$, we clearly have the following corollary of Theorem \ref{Thm:boundofmult}. 
\begin{cor}\label{deflesalpha}
Keep the same assumption and notations as in Theorem \ref{Thm:boundofmult}. Then $$\df_x(X, G)\leq \alpha_{d, e}(m).$$
\end{cor}

It is clear that $\alpha_{d,e}(m)\leq d-1$ if $m=\mult_xZ\geq 2$. Moreover,  $\alpha_{d,e}(m)$ is a decreasing function of $m$.

\begin{ex}\label{lemma:dim=2}
Assume that $Z$ is the minimal log canonical center of an effective divisor $G$ at $x\in X$ and  $\dim Z=2$. By Remark \ref{rmk:embedding-dim-surface-sing}, we know that $m=\mult_xZ=e-1$, where $e$ is the embedding dimension. Then  $\alpha_{2, e}(m)=1$ whenever $m\geq 2$.
\end{ex}

\begin{ex}\label{lemma:dim=3} Assume that $\dim X=5$ and $(X, G)$ is log canonical at $x\in X$ with the minimal log canonical center $Z$ at $x$. If $\dim Z=3$ and $Z$ is singular at $x$, then
\[\df_x(X, G)\leq\alpha_{3, 5}(m)=\begin{cases}2, ~~\mult_xZ=2, 3;\\ 1, ~~\mult_xZ=4, 5, 6.\end{cases}\]
\end{ex}

We remark that $\alpha_{d, e}(m)$ is sufficient to be used to prove Fujita's freeness conjecture on $4$-folds (see \cite{Ye2014}). For 5-folds, we need better bounds for deficits.  Suggested by Proposition \ref{def-up-new}, indeed we can obtain almost optimal upper bounds when the minimal log canonical center is a surface. 

\begin{lemma} \label{mld-rational-surface}
Let $(S, x)$ be a germ of rational surface singularity of multiplicity $m\leq 4$ and $\pi: \tilde{S}\to (S, x)$ be the minimal resolution of $x$. Then $\df_x(S,0)\leq\mld_x(S)\leq \frac{2}{m}$. Moreover, if $m=3$ and the minimal resolution of $(S, x)$ consists of at least one $(-2)$-curve, then $\df_x(S, 0)\leq \frac{3}{5}$. 
\end{lemma}

The proof of this lemma is elementary and only involves combinatorics of the fundamental cycle. We omit the proof for this lemma. 

\begin{rmk}
Lemma \ref{mld-rational-surface} suffices for our purpose in this paper. More generally, we also proved that $\df_x(S)\leq \frac{2}{m}$ and conjectured that $\mld_xS\leq \frac{2}{m}$ for any rational surface singularity with multiplicity $m$.  After a discussion with Jun Lu, he provided us a proof for the conjecture.  The authors thank him for allowing us to present his proof in Appendix \ref{sec:mld-def-surfaces}. 
 \end{rmk}

In the rest of the paper, we will let $\beta_G$ be a positive number such that $\beta_G\geq \df_x(X, G)$. For example, in some places in our proof,  we will take $$\beta_G:=\min\{\alpha_{d,e}(m),  \mld_x(Z(G), \diff_Z(G))\}.$$

\section{Increasing orders of vanishing \label{sec:increasing-order}}
Let $X$ be a smooth projective variety of dimension $n$ and $L$ be an ample line bundle on $X$. For any nonnegative real number $t$, we define  
\[\vol(t,L) :=\lim\limits_k\{\dfrac{n!}{k^n}\dim H^0(X, kL\otimes \mf{m}_x^{\ceil{kt}})\}\] to be the volume of the graded linear systems $\{ H^0(X, kL\otimes \mf{m}_x^{\ceil{kt}})\}$, where $x$ is closed point in $X$ and $\mf{m}_x$ is its maximal ideal.  For any two real numbers $\beta\leq \gamma$, we write $\vol(\beta,\gamma,L)=\vol(\beta,L)-\vol(\gamma,L)$.

We denote by $F_{k, q}(L)$ the fixed part of $|kL\otimes \mathfrak{m}_x^{\ceil{kq}}|$
and define \[\phi_k(q):=q-\frac{\ord_x F_{k, q}(L)}{k}\] if $|kL\otimes \mathfrak{m}_x^{\ceil{kq}}|\neq\emptyset$ and $\phi_k(q)=-\infty$ otherwise. We define $\phi(q):=\sup\limits_k\{\phi_k(q)\}.$ 

Recall the following results from \cite{Ye2014} which will be used in our proof.

\begin{prop}[{\cite[Proposition 3.3]{Ye2014}}]\label{optimizingboundfordivisor}
Let $L$ be an ample line bundle on $X$ with $L^n>\sigma^n\geq n^n$ and $x\in X$. Assume that for some rational number $q>\sigma$ the linear system $|kL\otimes \mathfrak{m}_x^{kq}|$ is nonempty for a sufficiently large divisible $k$. 
There exists an effective $\mathbb{Q}$--divisor $G$ linearly equivalent to $\lambda L$ for some positive number $\lambda<1$ such that it is critical (see Remark \ref{rmk:critical}) at $x$ with $\ord_x G=\lambda q>\lambda\sigma$.
Furthermore, if the minimal log canonical center $Z=M(G)$ is a divisor in $X$ with $\mult_xZ=m$ and  $\phi(q)>(n-m)-\mu(q-\sigma)$ for some positive number $\mu$, then 
\[\frac{\df_x G}{1-\lambda}<\frac{n-m+\mu\sigma}{1+\mu}.\]
\end{prop}

\begin{prop}[{\cite[Proposition 3.4]{Ye2014}}]\label{Best-mu}Assume that $L^n>\sigma^n\geq n^n$. For a real number $w$ with $0\leq w<n-1$, we set $\mu(w)$ be the minimal positivity number satisfying
 \begin{equation}\label{min-slope}
 \frac{(\frac{w}{\sigma}+\mu)^n}{\mu(1+\mu)^{n-1}}\leq 1.
 \end{equation}
There exist a rational number $q>\sigma$ such that
\begin{equation}\label{condonq}
\phi(q)> w-\mu(w)(q-\sigma)
\end{equation}
for all numbers $w\in [0, n-1)$. In particular, the linear system $|kL\otimes \mathfrak{m}_x^{kq}|$ is nonempty for a sufficiently large divisible $k$.
 \end{prop}
 
In our proof of the main theorem, we need better lower bound for the function $\phi(q)$ in Proposition  \ref{optimizingboundfordivisor}. By approximating $\phi(q)$ piecewisely, we obtain the following corollary of Proposition \ref{Best-mu}.
\begin{cor}\label{lowerboundphi}
In Proposition \ref{optimizingboundfordivisor}, we take $n=5$ and $\sigma=5.9999$. If $q>10$, then $\phi(q)>3-0.0391(q-\sigma)$. If $10\geq q>\sigma$,  then $\phi(q)>3.9999-0.2884(q-\sigma)$.
\end{cor}
\begin{proof}
If $q>10$, then \[\begin{split}\phi(q)&>\max\{3-0.0391(1-\sigma), 3.9999-0.2884(q-\sigma)\}\\
  &>3-0.0391(q-\sigma).
\end{split}\]
Similarly, if $10\geq q>\sigma$,  then \[\begin{split}\phi(q)&>\max\{3-0.0391(1-\sigma), 3.9999-0.2884(q-\sigma)\}\\
&>3.9999-0.2884(q-\sigma).
\end{split}\]
\end{proof}

In the proof of the main theorem, we need to calculate the volume of a graded linear system whose fixed part contains a prime Weil divisor. Due to the technicality and different flavor, we will discuss volume calculation along prime Weil divisors in Appendix \ref{sec:vol-Weil-div} which is weaker than the calculation in the Cartier case (see Lemma \ref{lengthdiv}).  In the following, we  set up some notations.

Let $(Z,\Delta)$ be a klt pair,  $A$ be an effective ample $\mb{Q}$--Cartier divisor on $Z$ and $x\in Z$ be a closed point. For each prime divisor $S\subset Z$ and rational numbers $t\geq 0$, we define a function 
\[\psi_S(t):=\lim_{k}\frac{\ord_S(F_{k, t}(A))}{k} \label{psi_X(t)}\]
where $k$ is taken to be sufficiently large and divisible, and $F_{k, t}(A)$ is the fixed part of the linear system $\lsyst{\O_X(kA)\otimes m_x^{kt}}$.  On $Z$, we define the volume in the same way as on smooth varieties:  \[\vol(t,A) :=\lim\limits_k\{\dfrac{n!}{k^n}\dim H^0(Z, \O_Z(kA)\otimes \mf{m}_x^{\ceil{kt}})\}.\]

\section{Cutting minimal log canonical centers \label{sec:cutting-center}}

In this section, we will  present some criteria for producing proper sub minimal log canonical centers.

\begin{lemma}[\cite{Ein1997}]\label{Siu-reduction-criterion}
Assume that $(X, G)$ is log canonical at  $x\in X$ with the minimal log canonical center $Z$ at $x$. Let $D$ be an  effecitve $\mathbb{Q}$--divisor on $X$ such that $$\ord_x(D|_Z)>\df_x(X, G).$$ Then there exists an divisor $G'=G+tD$, $t<1$, such that $(X, G')$ is also log canonical at $x$ and the minimal log canonical center $Z'$ of $(X, G')$ at $x$ is a proper subvariety of $Z$.
\end{lemma}

In practical, we have the following induction criterion due to Helmke which can be viewed as a consequence of the above lemma.

\begin{prop}[\cite{Helmke1997}]\label{Prop:Helmke'sInductionCriterion}
       Let  $L$ be an ample line bundle over $X$ and $G$ be an effective $\mathbb{Q}$--divisor linearly equivalent to $ \lambda L$ for some positive rational number $\lambda<1$. Assume that $(X, G)$ is log canonical at $x$ with $\df_x(X, G)$. Let $Z$ be the minimal log canonical center of $(X,G)$ at $x$ and $d=\dim Z>0$. If
    \begin{equation}\label{ineq}L^d\cdot Z>(\dfrac{\df_x(X, G)}{1-\lambda})^d\cdot \mult_x Z\end{equation}
then there is an effective $\mathbb{Q}$--divisor $G'$ linearly equivalent to $\lambda'L$ with $\lambda<\lambda'<1$ with the minimal log canonical center $Z'$ properly contained in $Z$ and  
$$\dfrac{\df_x(X, G')}{1-\lambda'}<\dfrac{\df_x(X, G)}{1-\lambda}.$$
\end{prop}

\begin{lemma}[\cite{Ye2014}]\label{main-lemma}
Assume that $\dim X=n$. Let   $G$ be an effective $\mathbb{Q}$--divisor on $X$, linearly equivalent to $\lambda L$ for some $\lambda<1$, such that $(X, G)$ is log canonical at  $x\in X$. If $\ord_xG\geq \lambda\sigma$ for some $\sigma>n$, then 
\[ \dfrac{\df_x(X, G)}{1-\lambda}\leq \dfrac{\sigma \beta_G}{\sigma-n+\beta_G},\]
for any $\beta_G\geq \df_x(X,G)$.

In particular, if $L^n>\sigma^n\geq  n^n$, then there exists an effective $\mathbb{Q}$--divisor $G$ linearly equivalent to $\lambda L$ with $\lambda<1$ such that $(X, G)$ is log canonical at $x$ and
\[ \dfrac{\df_x(X, G)}{1-\lambda}< \dfrac{\sigma \beta_G}{\sigma-n+\beta_G}.\] \end{lemma}

The same idea used in the proof of the above lemma together with Lemma \ref{lemma:Ein-Siu} leads to the following result for further inductions steps in special situations.

\begin{lemma}\label{secondreductioncriterion} Assume that $\dim X=n$.
Let  $L$ be an ample line bundle on $X$ and $x\in X$. Assume that $\sqrt[n]{L^n}> \sigma \geq n$ and $\sqrt[d]{L^d\cdot Z}> \sigma$ for $d=1, \dots, n-1$.  Let $G$ be an effective divisor linearly equivalent to $\lambda L$ with $\lambda<1$.  Assume that $(X, G)$ is log canonical at $x$ with the minimal log canonical center $Z=M(G)$ whose dimension is $d>0$. Denote by $t=\df_x(X, G)\frac{\sqrt[d]{m}}{\sigma}$.  Assume $\frac{\df_x(X, G)}{1-\lambda}<\frac{\sigma}{\sqrt[d]{m}}$.
\begin{enumerate}[ref=\alph*]
\item \label{secondreductioncriterion-part-1}
There exist an effective $\mathbb{Q}$--divisor $D$ linearly equivalent to $tL$  and a positive number $c<1$ such that $(X, G+cD)$ is log canonical at $x$ whose minimal log canonical center $W=M(G+cD)$ is a proper subvariety of $Z$.
\item Let $G'=G+cD$ be the divisor in \eqref{secondreductioncriterion-part-1}. Let  $\lambda'=\lambda+ct$,  $m'=\mult_xW$ and let $\beta_{G'}$ be a number such that $\df_x(X, G')\leq \beta_{G'}$. Then  $G'$ is linearly equivalent to $\lambda'L<1$ and
\begin{eqnarray}\label{estofq}\dfrac{\df_x(X, G')}{1-\lambda'}\leq\begin{cases}
\dfrac{\beta_{G'}}{(1-\lambda) -\df_x(X, G)\frac{\sqrt[d]{m}}{\sigma} + \beta_{G'}\frac{\sqrt[d]{m}}{\sigma}} &  \df_x(X, G)>\beta_{G'}\\
\dfrac{\df_x(X, G)}{1-\lambda} & \df_x(X, G)\leq \beta_{G'}.
\end{cases}
\end{eqnarray}

\item If  $\ord_xG\geq \lambda\sigma$, then we have
\begin{equation}\label{intermediate-steps-smooth}
\dfrac{def_x(X, G')}{1-\lambda'}\leq
\begin{cases} \dfrac{\beta_{G'}}{\frac{\sigma-n+\beta_G}{\sigma}-\frac{\sqrt[d]{m}(\beta_G-\beta_{G'})}{\sigma}}& \df_x(X, G)> \beta_{G'} \\
\dfrac{\sigma\beta_{G'}}{\sigma-n+\beta_{G'}} &  \df_x(X, G)\leq \beta_{G'}
\end{cases}
\end{equation}
\item If $\ord_xG\geq\lambda\sigma$ and the minimal log canonical center $Z$ is smooth at $x$. Then
\begin{equation}\label{all-steps-smooth}
\dfrac{def_x(X, G')}{1-\lambda'}\leq \frac{\sigma\beta_{G'}}{\sigma-n+\beta_{G'}}. \end{equation}
\end{enumerate}
\end{lemma}

Here the number $c$ in this lemma is the log canonical threshold of the pair $(X,G)$ against $D$ at $x$. For simplicity, we call it the {\em log canonical threshold of the triple} $(X, G, D)$ at $x$.

\begin{proof} 
For any sufficiently small positive number  $\varepsilon$,
we can find a new effective $\mathbb{Q}$--divisor $D\supsetneq Z$  linearly equivalent to $tL$ with $\ord_xD|_{Z}=\df_x(X, G)+\varepsilon$. Let $c\in(0, 1)$ be the minimal number such that $(X, G+cD)$ is log canonical at $x$. Hence $G'=G+cD$ is linearly equivalent to $\lambda'$. The assumption $\frac{\df_x(X,G)}{1-\lambda}<\frac{\sigma}{\sqrt[d]{m}}$ implies that $\lambda'<1$. By Lemma \ref{Siu-reduction-criterion} and Lemma \ref{lemma:Ein-Siu},  we know that $W=M(G')$ is a proper subvariety of $Z$ and  \begin{equation}\label{eq:bound-c}\df_x(X, G')\leq \min\{\beta_{G'}, \df_x(X, G)-c\cdot \ord_xD|_Z\}\leq \min\{\beta_{G'}, (1-c)\cdot\df_x(X, G) \}.\end{equation}
The proof of the inequality \eqref{estofq} is divided into two cases.
\begin{enumerate}[label=(\Roman*)]
\item Assume that  $\beta_{G'}\geq (1-c)\df_x(X, G)$, equivalently $c\geq 1-\frac{\beta_{G'}}{\df_x(X, G)}$. Then
 \begin{equation}\label{eq:p(c)} \dfrac{\df_x(X, G')}{1-\lambda'} \leq \dfrac{ (1-c)\cdot\df_x(X, G) }{1-\lambda-ct}=\dfrac{ (1-c)\cdot\df_x(X, G) }{(1-\lambda)-c\cdot\df_x(X, G)\frac{\sqrt[d]{m}}{\sigma}}=:p(c).\end{equation}
 Note that $p(c)$ is a decreasing function of $c$. If $\df_x(X, G)>\beta_{G'}$, then $c\geq 1-\frac{\beta_{G'}}{\df_x(X, G)}>0.$  Consequently, we have
 \[p(c)\leq \dfrac{\beta_{G'}}{(1-\lambda) - \df_x(X, G)\frac{\sqrt[d]{m}}{\sigma}+\frac{\beta_{G'}\sqrt[d]{m}}{\sigma}}.\]
If $\df_x(X, G)\leq \beta_{G'}$, then $c\geq 0 \geq 1-\frac{\beta_{G'}}{\df_x(X, G)}$ and
\[ \dfrac{\df_x(X, G')}{1-\lambda'} \leq p(c)\leq \dfrac{\df_x(X, G)}{1-\lambda}. \]

\item Assume that  $\beta_{G'}\leq (1-c)\cdot\df_x(X, G)$, equivalently  $c\leq 1-\frac{\beta_{G'}}{\df_x(X, G)}$, which is possible only under the assumption that $\df_x(X, G)>\beta_{G'}$  Then
 \begin{equation}\label{eq:q(c)} \dfrac{\df_x(X, G')}{1-\lambda'} \leq \dfrac{\beta_{G'}}{1-\lambda-ct}=:q(c).\end{equation}
 Note that $q(c)$ is an increasing function of $c$. Then we have
 \[q(c)\leq \dfrac{\beta_{G'}}{(1-\lambda) - \df_x(X, G)\frac{\sqrt[d]{m}}{\sigma}+\frac{\beta_{G'}\sqrt[d]{m}}{\sigma}}\]
Therefore,
\[\dfrac{\df_x(X, G')}{1-\lambda'}\leq\dfrac{\beta_{G'}}{(1-\lambda) -\df_x(X, G)\frac{\sqrt[d]{m}}{\sigma} + \beta_{G'}\frac{\sqrt[d]{m}}{\sigma}}.
\]
\end{enumerate}

Now we prove the inequality \eqref{intermediate-steps-smooth}. Assume that $\ord_xG\geq \lambda\sigma$. Then $$0\leq \df_x(X, G)\leq \min\{\beta_G, n-\lambda\sigma\}.$$ If $\df_x(X, G)\leq \beta_{G'}$, then the conclusion follows from the inequality \eqref{estofq} and Lemma \ref{main-lemma}.

We  assume that $\df_x(X, G)>\beta_{G'}$.  Write
\[r(\lambda, \df_x(X, G)):=1-\lambda -\df_x(X, G)\frac{\sqrt[d]{m}}{\sigma}.\]
If $\beta_G\leq n-\lambda\sigma$, equivalently $\lambda\leq \frac{n-\beta_G}{\sigma}$, then
\begin{eqnarray*}r(\lambda, \df_x(X, G))&\geq& 1-\lambda - \beta_G\frac{\sqrt[d]{m}}{\sigma}
\geq 1-\frac{n-\beta_G}{\sigma}-\beta_G\frac{\sqrt[d]{m}}{\sigma}\\
&=&\frac{\sigma-n+\beta_G}{\sigma}-\frac{\sqrt[d]{m}\beta_G}{\sigma}.
\end{eqnarray*}
If $\beta_G\geq n-\lambda\sigma$, equivalently $\lambda\geq \frac{n-\beta_G}{\sigma}$, then
\begin{eqnarray*}r(\lambda, \df_x(X, G))&\geq& 1-\lambda - (n-\lambda\sigma)\frac{\sqrt[d]{m}}{\sigma}=(1-n\frac{\sqrt[d]{m}}{\sigma})+(\sqrt[d]{m}-1)\lambda\\
&\geq&  (1-n\frac{\sqrt[d]{m}}{\sigma})+(\sqrt[d]{m}-1)\frac{n-\beta_G}{\sigma} \\ &=&\frac{\sigma-n+\beta_G}{\sigma}-\frac{\sqrt[d]{m}\beta_G}{\sigma}
\end{eqnarray*}  By the inequality \eqref{estofq}, we know that
\[\dfrac{\df_x(X, G')}{1-\lambda'}\leq  \dfrac{\beta_{G'}}{\frac{\sigma-n+\beta_G}{\sigma}-\frac{\sqrt[d]{m}(\beta_G-\beta_{G'})}{\sigma}}.\]

Assume in addition that $Z$ is smooth at $x$. The inequality \eqref{all-steps-smooth} follows from  the inequality \eqref{intermediate-steps-smooth} by plugging $m=1$ into it.  
\end{proof}

\section{Proof of the main theorem \label{sec:main-proof}}
Let $\sigma=5.9999$.  By Proposition \ref{Best-mu} and \ref{optimizingboundfordivisor}, there exists an effective $\mathbb{Q}$--divisor $G$ linearly equivalent to $\lambda L$ with $0<\lambda<1$ such that $(X, G)$ is log canonical at $x$ and $\ord_xG=\lambda q>\lambda\sigma$. Denote by $Z=M(G)$ the minimal log canonical center of $(X, G)$ at $x$.

We will prove the main theorem by descending induction on the dimension of log canonical centers. By Proposition \ref{Prop:Helmke'sInductionCriterion}, if there is a lc pair $(X, G)$ with the minimal log canonical center $Z$ satisfying the following property
\begin{equation}\label{eq:punch-line}
\dfrac{\df_x(X, G)}{1-\lambda}<\min\limits_{Z'\subseteq Z}\{\frac{6}{\sqrt[\dim Z']{\mult_xZ'}}\}
\end{equation} where $Z'$ is the minimal log canonical center of a lc pair $(X, P)$ at $x$ and $P$ is an effective $\mathbb{Q}$--divisor on $X$, then we can repeatedly apply Proposition \ref{Prop:Helmke'sInductionCriterion} and then Lemma \ref{cor:0-dimb-p-f} to complete the proof. To simplify the proof, we aim at showing that \eqref{eq:punch-line} holds. 

The proof of our main theorem will run though all possible dimensions and singularities of the minimal log canonical center $M(G)$. We divide the proof into three parts according to the complexity of arguments. In the first part, we consider easy cases: $\dim Z\leq 2$, or $\mult_xZ=1$, or $\dim Z=4$ and $\mult_xZ\geq 3$. In the second part, we consider the case that $\dim Z=4$ and $\mult_xZ=2$.  The most complicated case that $Z$ is singular of dimension $3$ will be proved in the last part. 

\begin{proof}[Proof of Main Theorem]\mbox{}

\subsection{Easy cases.}\mbox{}

\case{Assume that  $Z$ is a curve.}

Since $\mult_xZ=1$, then $\df_x(X, G)\leq\beta_G:=\dim Z=1$. By Lemma \ref{main-lemma}, we know that \[\dfrac{\df_x(X, G)}{1-\lambda}\leq \dfrac{\sigma}{\sigma-5+1}<6.\] 

\case{Assume that  $Z$ is a surface.} 

In this case, we know that $\mult_xZ\leq 4$ and
\[\df_x(X, G)\leq \begin{cases}2&\mult_xZ=1\\
1&\mult_xZ\geq 2.\end{cases}\]

Assume that $\mult_xZ\leq 3$. Apply Lemma \ref{main-lemma}, we get
 \[\dfrac{\df_x(X, G)}{1-\lambda}\leq \begin{cases}\dfrac{2\sigma}{\sigma-5+2}&\mult_xZ=1\\
 \dfrac{\sigma}{\sigma-5+1}&\mult_xZ=2, 3
 \end{cases}<\min\{\frac{6}{\sqrt{\mult_xZ}}, 6\}.\]

Assume that $\mult_xZ=4$. Then $\mld_xZ\leq \frac{1}{2}$ by Lemma \ref{mld-rational-surface}. By Proposition \ref{prop:Ein-Helmke}, we see that 
\[\df_x(X, G)\leq \df_x(Z,\diff_Z(G))\leq \mld_x(Z)\leq \frac{1}{2}=:\beta_G.\] Therefore,
\[\dfrac{\df_x(X, G)}{1-\lambda}\leq \dfrac{\frac{1}{2}\sigma}{\sigma-5+\frac{1}{2}}<2.1<\min\{\frac{6}{\sqrt{4}}, 6\}.\]

\case{Assume that  $Z$ is smooth and $\dim Z=3$.}

In this case, we know that $\beta_G\leq 3$. By Lemma \ref{main-lemma}  and Proposition \ref{Prop:Helmke'sInductionCriterion}, we can find a new divisor $G_1$ linearly equivalent to $\lambda_1L$ with  $\ord_xG_1\geq \lambda_1\sigma$ such that $(X, G_1)$ is log canonical at $x$ and the minimal log canonical center $Z_1=Z(G_1)$ is properly contained in $Z$. Moreover, we have
\[\frac{\df_x(X,G_1)}{1-\lambda_1}\leq\frac{\sigma \beta_{G_1}}{\sigma-n+\beta
_{G_1}}\]
by Lemma \ref{secondreductioncriterion}.

If $Z_1$ is a curve, we take $\beta_{G_1}=1$, then
\[\frac{\df_x(X,G_1)}{1-\lambda_1}\leq\frac{\sigma}{\sigma-5+1}<6.\]

If $Z_1$ is a surface, then the embedding dimension of $x\in Z_1$ is at most $3$ and $\mult_xZ_1\leq 2$. Take $\beta_{G_1}=\frac{2}{\mult_xZ_1}$ for $\mult_xZ_1=1, 2$, we see that 
 \[\frac{\df_x(X,G_1)}{1-\lambda_1}\leq \frac{\sigma \beta_{G_1}}{\sigma-5+\beta_{G_1}}<\min\{\dfrac{6}{\sqrt{\mult_xZ_1}},6\}.\]

\case{Assume that $Z$ is a smooth divisor.}

Similar to the case that $Z$ is a smooth threefold, 
By Lemma \ref{secondreductioncriterion}, we can find a new divisor $G_1$ linearly equivalent to $\lambda_1L$ with $\ord_xG_1\geq \lambda_1\sigma$ such that $(X, G_1)$ is log canonical at $x$ and the minimal log canonical center $Z_1=Z(G_1)\subsetneq Z$. Moreover, 
\[\frac{\df_x(X,G_1)}{1-\lambda_1}< \frac{\sigma\beta_{G_1}}{\sigma-5+\beta_{G_1}}.\]
The embedding dimension of $x\in Z_1$ is at most $4$, hence $m_1:=\mult_xZ_1\leq 3$. Let
\[
\beta_{G_1}=
\begin{cases}
3&\dim Z_1=3, m_1=1\\
2& \dim Z_1=3, m_1=2 \text{~or~} \dim Z_1=2, \mult_xZ_1=1\\
1& \dim Z_1=3, m_1=3 \text{~or~} \dim Z_1=2, \mult_xZ_1=2, 3 \text{~or~} \dim Z_1=1.
\end{cases}
\]
We have
\begin{equation}
\begin{split}\frac{\df_x(X,G_1)}{1-\lambda_1}&< \frac{\sigma\beta_{G_1}}{\sigma-5+\beta_{G_1}}\\
&\leq {\begin{cases}4.6& \dim Z_1=3, \mult_xZ_1=1\\
                   4.1& \dim Z_1=3, m_1=2 \text{~or~} \dim Z_1=2, \mult_xZ_1=1\\
                   3.1& \begin{split}
                              &\dim Z_1=3, m_1=3 \text{~or~} \dim Z_1=2, \mult_xZ_1=2, 3\\ 
                              &\text{~or~} \dim Z_1=1\end{split}
                           \end{cases} }\\
&<\min\{\dfrac{6}{\sqrt[\dim Z_1]{\mult_xZ_1}}, 6\}.\label{eq:smooth-div-step-1-ineq}
\end{split}\end{equation}
Then by Proposition \ref{Prop:Helmke'sInductionCriterion}, we can find an new effective divisor $G_2$ linearly equivalent to $\lambda_2L$ with  $0<\lambda_2<1$ such that $(X, G_2)$ is log canonical at $x$ and the minimal log canonical center $Z_2$ is properly contained in $Z_1$. By inequality \eqref{eq:smooth-div-step-1-ineq}, we see that only the case that $\dim Z_1=3$ and $\dim Z_2=2$ requires further arguments. 

 Assume that $\dim Z_1=3$, $\mult_xZ_1=1$ and $\dim Z_2=2$. Then $\mult_xZ_2\leq 2$.  If  $Z_2$ is also smooth, then we are done by inequality \eqref{eq:smooth-div-step-1-ineq}.  If $\mult_xZ_2=2$, then $\beta_{G_2}\leq 1$.  Apply Lemma \ref{secondreductioncriterion} to $(X, G_1)$, we have
 \[\frac{\df_x(X,G_2)}{1-\lambda_2}<\frac{\sigma\beta_{G_2}}{\sigma-5+\beta_{G_2}}<3.1<\dfrac{6}{\sqrt{2}}.\]
 
Assume that $\dim Z_1=3$, $\mult_x{Z_1}=2$, $\dim Z_2=2$ and $\mult_xZ_2=3$, then $\beta_{G_1}=2$ and $\beta_{G_2}\leq \frac{2}{3}$. Again we apply Lemma \ref{secondreductioncriterion} to $(X, G_1)$ and obtain that
\[\frac{\df_x(X, G_2)}{1-\lambda_2}\leq \dfrac{\frac{2}{3}}{\frac{\sigma-5+2}{\sigma}-\frac{\sqrt[3]{2}(2-\frac{2}{3})}{6}}<3.1<\dfrac{6}{\sqrt{3}}.\] 

Assume in the last that $\dim Z_1=3$ and $\mult_xZ_1=3$. We have \[\frac{\df_xG_2}{1-\lambda_2}<\frac{\df_xG_1}{1-\lambda_1}<3.1<\frac{6}{\sqrt{3}}.\]

\case{Assume that $Z$ is a divisor and $\mult_xZ\geq 3$.\label{sec:divisor-center-m>2}}

 In Proposition \ref{Best-mu}, taking $w=2, 1$ for $\mult_xZ=3, 4$ respectively and solving for $\mu(w)$, we get $\mu(2)<0.0044$ and $\mu(1)<0.0002$.  Apply Proposition \ref{optimizingboundfordivisor}, we have
\[\frac{\df_x(X, G)}{1-\lambda}<\frac{(5-m)+\mu(w)\sigma}{1+\mu(w)}<2.1< \min\{\frac{6}{\sqrt{4}}, \frac{6}{\sqrt[3]{6}},\frac{6}{\sqrt[4]{4}}\}.\]

\subsection{Divisor center with multiplicity $2$.\label{codimm=2}}\mbox{}

Assume that $\dim Z=4$ and $\mult_xZ=2$. Take $w=3$, we get $\mu(3)<0.0391$. By Proposition \ref{optimizingboundfordivisor}, we have 
\[\frac{\df_x(X, G)}{1-\lambda}<\frac{3+\mu(3)\sigma}{1+\mu(3)}<3.12<\dfrac{6}{\sqrt[4]{2}},\] 
 By Proposition  \ref{Prop:Helmke'sInductionCriterion}, we can find a new divisor $G_1$ such that $(X, G_1)$ is log canonical at $x$ and the minimal log canonical center $Z_1=M(G_1)$ is properly contained in $Z$.  
If $Z_1$ is smooth, then 
 \[\frac{\df_x(X, G_1)}{1-\lambda_1}<\frac{\df_x(X, G)}{1-\lambda}<3.12< \min\{6, \dfrac{6}{\sqrt{2}}\}.\]  
 Therefore, our induction can be carried on till we get a $0$-dimensional minimal log canonical center. 
 Now we assume that $Z_1$ is singular. It is possible that $Z_1$ is a surface with $\mult_xZ_1=4$ or that $Z_1$ is of dimension $3$ and contains a surface $Z_2$ with $\mult_xZ_2=4$ as a new minimal log canonical center.  We now show that \[\frac{\df_x(X, G_1)}{1-\lambda_1}<3=\min\{\dfrac{6}{\sqrt{4}}, \dfrac{6}{\sqrt[3]{6}}\}.\] Hence the induction can be carried on.

\begin{lemma}\label{5-folds-secondreduction:first=divisor}
If the minimal log canonical center $Z=M(G)$ is a divisor with $\mult_x(Z)=2$, then there exists a new divisor $G_1=\lambda_1L$ with $\lambda<\lambda_1<1$ such that $(X, G_1)$ is log canonical at $x$, the minimal log canonical center $Z_1=M(G_1)$ is properly contained in $Z$ and
\[\frac{\df_x(X, G_1)}{1-\lambda_1}<3\]
if $Z_1$ is singular.
\end{lemma}
\begin{proof}
The existence of $Z_1$ is clear.  We focus on proving the inequality. Since $Z_1$ is singular at $x$ and of dimension $3$ or $2$, we can take $\beta_{G_1}=2$.  By the proof of \cite[Proposition 3.3]{Ye2014}, Corollary \ref{lowerboundphi} and Proposition \ref{Ein-Helmke-def-up-first}, we see that  \[\df_x(X, G)<\beta_{G}:=\min\{5-\lambda q, 3-\lambda(3-0.0391(q-\sigma)),3-\lambda(3.9999-0.2884(q-\sigma))\}.\]

We imitate the proof of Lemma \ref{secondreductioncriterion} with improvement on $\df_x(X, G)$.  Recall the following notation from the proof of Lemma \ref{secondreductioncriterion}: \[r(\lambda, \df_x(X, G))=1-\lambda-\df_x(X, G)\frac{\sqrt[d]{m}}{\sigma}.\] 

\case{Assume $q>10$ and $\df_x(X, G)>2$.} 

Since $q>10$, then by the proof of Corollary \ref{lowerboundphi} we see that \[\beta_{G}=\min\{5-\lambda q, 3-\lambda(3-0.0391(q-\sigma))\}.\]
\begin{enumerate}
\item
If $3-\lambda(3-0.0391(q-\sigma))\geq5-\lambda q>2$, then $\frac{3}{q}>\lambda\geq \frac{2}{1.0391q-3.23459609}$ and $\beta_G=5-\lambda q$. It follows that
\[\begin{split}r(\lambda, \df_x(X, G))&\geq 1-\lambda-(5-\lambda q)\frac{\sqrt[4]{2}}{\sigma}\\
&\geq \frac{\sigma-5\sqrt[4]{2}}{\sigma}+\frac{q\sqrt[4]{2}-\sigma}{\sigma}\lambda\\
&\geq \frac{\sigma-5\sqrt[4]{2}}{\sigma}+\frac{q\sqrt[4]{2}-\sigma}{\sigma}\frac{2}{1.0391q-3.23459609}\\
&> \frac{\sigma-5\sqrt[4]{2}}{\sigma}+\frac{10\sqrt[4]{2}-\sigma}{\sigma}\frac{2}{1.0391\cdot 10-3.23459609}\\
&>0.2834.
\end{split}\]
By inequality \eqref{estofq} in Lemma \ref{secondreductioncriterion}, we get $\frac{\df_x(X, G_1)}{1-\lambda_1}<\frac{2}{0.2834+\frac{2\sqrt[4]{2}}{\sigma}}<2.86$.

\item
If  $5-\lambda q\geq  3-\lambda(3- 0.0391(q-\sigma))>2$, then
\[\lambda\leq \frac{2}{1.0391q-3.23459609}\]
and $\beta_G=3-\lambda(3- 0.0391(q-\sigma))$. 
It follows that
\[\begin{split}r(\lambda, \df_x(X, G))
&\geq 1-\lambda-( 3-\lambda(3- 0.0391(q-\sigma)))\frac{\sqrt[4]{2}}{\sigma}\\
&\geq 1-\frac{3\sqrt[4]{2}}{\sigma}+\lambda((3- 0.0391(q-\sigma))\frac{\sqrt[4]{2}}{\sigma}-1)\\
&\geq 1-\frac{3\sqrt[4]{2}}{\sigma}+\frac{2((3- 0.0391(q-\sigma))\frac{\sqrt[4]{2}}{\sigma}-1)}{1.0391q-3.2339961}\\
&=(\frac{\sigma-5\sqrt[4]{2}}{\sigma})+\frac{q\sqrt[4]{2}-\sigma}{\sigma}\cdot\frac{2}{1.0391q-3.23459609}\\
&> (\frac{\sigma-5\sqrt[4]{2}}{\sigma})+\frac{10\sqrt[4]{2}-\sigma}{\sigma}\cdot\frac{2}{1.0391\cdot 10-3.23459609}\\
&>0.2834.
\end{split}\]
Therefore, $\frac{\df_x(X, G_1)}{1-\lambda_1}<\frac{2}{0.2834+\frac{2\sqrt[4]{2}}{\sigma}}<2.86$.
\end{enumerate}

\case{Assume that $q>10$ and $\df_x(X, G)\leq 2$.}

 We know that  \[\df_x(X, G)\leq \min\{2, 5-\lambda q\} .\]
\begin{enumerate}
\item 
If $5-\lambda q\geq 2$, then $\lambda\leq \frac{3}{q}\leq \frac{3}{10}$.
It follows that
\[\dfrac{\df_x(X, G_1)}{1-\lambda_1}<\dfrac{\df_x(X, G)}{1-\lambda}\leq \frac{2}{1-\lambda}\leq \frac{20}{7}.\]

\item
 If $5-\lambda q\leq  2$, then $\lambda\geq \frac{3}{q}$.
It follows that
\[\dfrac{\df_x(X, G_1)}{1-\lambda_1}<\dfrac{\df_x(X, G)}{1-\lambda}\leq \frac{5-\lambda q}{1-\lambda}\leq \frac{2}{1-\frac{3}{q}}\leq \frac{20}{7}.\]
\end{enumerate}

\case{Assume that $q\leq 10$ and $\df_x(X, G)>2$.}

By the proof of Corollary \ref{lowerboundphi}, we know that
\[\beta_{G}=\min\{5-\lambda q,,3-\lambda(3.9999-0.2884(q-\sigma))\}.\]
\begin{enumerate}
\item
If $ 3-\lambda(3.9999-0.2884(q-\sigma))\geq 5-\lambda q\geq 2$, then $\frac{3}{q}> \lambda\geq \frac{2}{1.2884q-5.73027116}$ which implies that $q> 9.2$.
Then
\[\begin{split}r(\lambda, \df_x(X, G))&\geq 1-\lambda-(5-\lambda q)\frac{\sqrt[4]{2}}{\sigma}\\
&\geq (\frac{\sigma-5\sqrt[4]{2}}{\sigma})+\frac{q\sqrt[4]{2}-\sigma}{\sigma}\lambda\\
&\geq (\frac{\sigma-5\sqrt[4]{2}}{\sigma})+\frac{q\sqrt[4]{2}-\sigma}{\sigma}\frac{2}{1.2884q-5.73027116}\\
&> (\frac{\sigma-5\sqrt[4]{2}}{\sigma})+\frac{9.2\sqrt[4]{2}-\sigma}{\sigma}\frac{2}{1.2884\cdot 9.2-5.73027116}\\
&>0.2779.
\end{split}\]
It follows that $\frac{\df_x(X, G_1)}{1-\lambda_1}<\frac{2}{0.2779+\frac{2\sqrt[4]{2}}{\sigma}}<2.97$.

\item 
If  $5-\lambda q\geq  3-\lambda(3.9999- 0.2884(q-\sigma))>2$, then 
\[\lambda\leq \min\{\frac{2}{1.2884q-5.73027116}, \frac{1}{5.73027116-0.2884q}\}.\] It follows that  \[\begin{split}r(\lambda, \df_x(X, G))
&\geq 1-\lambda-(3-\lambda(3.9999- 0.2884(q-\sigma)))\frac{\sqrt[4]{2}}{\sigma}\\
&= 1-\frac{3\sqrt[4]{2}}{\sigma}+\lambda((3.9999- 0.2884(q-\sigma))\frac{\sqrt[4]{2}}{\sigma}-1).\end{split}\]

\begin{enumerate}
\item[(1)] Assume that $\frac{2}{1.2884q-5.73027116}<\frac{1}{5.73027116-0.2884q}$. Then $q>9.2$ and
\[\begin{split}r(\lambda, \df_x(X, G))
&\geq 1-\frac{3\sqrt[4]{2}}{\sigma}+\frac{2((3.9999- 0.2884(q-\sigma))\frac{\sqrt[4]{2}}{\sigma}-1)}{1.2884q-5.73027116}\\
&=(\frac{\sigma-5\sqrt[4]{2}}{\sigma})+\frac{2}{\sigma}\cdot\frac{q\sqrt[4]{2}-\sigma}{1.2884q-5.73027116}\\
&> (\frac{\sigma-5\sqrt[4]{2}}{\sigma})+\frac{2}{\sigma}\cdot\frac{9.2\sqrt[4]{2}-\sigma}{1.2884\cdot 9.2-5.73027116}\\
&>0.2779.
\end{split}\] Consequently,  \[\frac{\df_x(X, G_1)}{1-\lambda_1}<\frac{2}{0.2779+\frac{2\sqrt[4]{2}}{\sigma}}<2.97.\]
\item[(2)] Assume that $\frac{2}{1.2884q-5.73027116}\geq \frac{1}{5.73027116-0.2884q}$. Then $q<9.22$ and
\[\begin{split}r(\lambda, \df_x(X, G))
&\geq 1-\frac{3\sqrt[4]{2}}{\sigma}+\frac{(3.9999- 0.2884(q-\sigma))\frac{\sqrt[4]{2}}{\sigma}-1}{5.73027116-0.2884q}\\
&\geq 1-\frac{2\sqrt[4]{2}}{\sigma}+ \frac{1}{0.2884q-5.73027116}\\
&>1-\frac{2\sqrt[4]{2}}{\sigma}+ \frac{1}{0.2884\cdot 9.22-5.73027116}.
\end{split}\]
Consequently, 
\[\frac{\df_x(X, G_1)}{1-\lambda_1}<\frac{2}{1+\frac{1}{0.2884\cdot 9.22-5.73027116}}<2.97.\]
\end{enumerate}
\end{enumerate}

\case{Assume that  $q<10$ and $\df_x(X, G)\leq 2$.}

By Lemma \ref{secondreductioncriterion}, we know that $\frac{\df_x(X, G_1)}{1-\lambda_1}<\frac{\df_x(X, G)}{1-\lambda}$.  

\begin{enumerate}
\item If $2\leq 5-\lambda q\leq  3-\lambda(3.9999- 0.2884(q-\sigma))$, then $\frac{3}{q}\geq \lambda\geq \frac{2}{1.2884q-5.73027116}$ which implies that $q> 9.2$. Thus,
\[\frac{\df_x(X, G)}{1-\lambda}\leq \frac{2}{1-\lambda}\leq \frac{2q}{q-3}<\frac{2\cdot 9.2}{9.2-3}<2.97.\]

\item If $2\leq 3-\lambda(3.9999- 0.288(q-\sigma))\leq 5-\lambda q$, then
$$\lambda\leq \min\{\frac{2}{1.2884q-5.73027116}, \frac{1}{5.73027116-0.2884q}\}.$$
\begin{enumerate}
\item[(1)] Assume that $\frac{2}{1.2884q-5.73027116}<\frac{1}{5.73027116-0.2884q}$. Then $q>9.2$ and \[\frac{\df_x(X, G)}{1-\lambda}\leq \frac{2}{1-\lambda}\leq \frac{2}{1-\frac{2}{1.2884q-5.73027116}}<\frac{2}{1-\frac{2}{1.2884\cdot 9.2-5.73027116}}<2.98.\]
\item[(2)]
Assume that $\frac{2}{1.2884q-5.73027116}\geq \frac{1}{5.73027116-0.2884q}$. Then $q<9.22$ and
\[\frac{\df_x(X, G)}{1-\lambda}\leq \frac{2}{1-\lambda}\leq \frac{2}{1-\frac{1}{5.73027116-0.2884q}}<\frac{2}{1-\frac{1}{5.73027116-0.2884\cdot 9.22}}<2.97.\]
\end{enumerate}

\item 
If $5-\lambda q\leq 2\leq  3-\lambda(3.9999- 0.2884(q-\sigma))$,  then $\frac{3}{q}\leq \lambda\leq \frac{1}{5.73027116-0.2884q}$ which implies that $q>9.2$. Therefore,
\[\frac{\df_x(X, G)}{1-\lambda}\leq \frac{5-\lambda q}{1-\lambda}\leq \frac{2q}{q-3}<\frac{2\cdot 9.2}{9.2-3}<2.97.\]

\item 
If $5-\lambda q\leq  3-\lambda(3.9999- 0.2884(q-\sigma))\leq 2$, then  $$\lambda\geq \max\{\frac{2}{1.2884q-5.73027116}, \frac{1}{5.73027116-0.2884q}\}.$$ If $\frac{2}{1.2884q-5.73027116}\leq \frac{1}{5.73027116-0.2884q}$, then $q>9.2$. Therefore,
\[\begin{split}\frac{\df_x(X, G)}{1-\lambda}&\leq \frac{5-\lambda q}{1-\lambda}\\ &\leq \frac{5-\frac{q}{5.73027116-0.2884q}}{1-\frac{1}{5.73027116-0.2884q}}\\ &= \frac{5\cdot 5.73027116-(1+5\cdot0.2884)\cdot 9.2}{4.73027116-0.2884\cdot9.2} <2.98.
\end{split}\]

\item 
If $ 3-\lambda(3.9999- 0.2884(q-\sigma))\leq 2\leq 5-\lambda q$,  then $\frac{1}{5.73027116-0.2884q}\leq  \lambda\leq \frac{3}{q}$ which implies that $q<9.22$.  Therefore,
\[\begin{split}\frac{\df_x(X, G)}{1-\lambda}&\leq \frac{3-\lambda(5.7278712-0.288q)}{1-\lambda}\\
&\leq \frac{3-\frac{1}{5.73027116-0.2884q}(5.73027116-0.2884q)}{1-\frac{1}{5.73027116-0.2884q}} \\
&\leq \frac{2}{1-\frac{1}{5.73027116-0.2884\cdot9.2}}<2.97
\end{split}\]

\item 
If $3-\lambda(3.9999- 0.2884(q-\sigma))\leq 5-\lambda q \leq 2$, then $ \frac{3}{q} \leq \lambda \leq \frac{2}{1.2884q-5.73027116}$ which implies that $q<9.22$.  Therefore,
\[\begin{split}\frac{\df_x(X, G)}{1-\lambda}&\leq \frac{3-\lambda(3.9999- 0.2884(q-\sigma))}{1-\lambda}\\
&\leq \frac{3-\frac{3}{q}(5.73027116-0.2884q)}{1-\frac{3}{q}} \\
&< \frac{3\cdot 9.22-3(5.73027116-0.2884\cdot 9.22)}{9.22- 3}<2.97.
\end{split}\]
\end{enumerate}

From the above case-by-case argument, we know that $\frac{\df_x(X, G_1)}{1-\lambda_1}<3$
if $Z_1$ is singular.
\end{proof}

\subsection{3-dimensional singular centers. \label{subsec:most-difficult}}\mbox{}

Assume that the minimal log canonical center $Z$ of $(X, G)$ at $x$ has dimension $3$ and $\mult_xZ\geq 2$.
\case{Assume that $\mult_xZ\geq 4$.\label{case:3-fold-m>=4}} 

 In this cace, $\df_x(X,G)\leq \alpha_{3, 5}(m)\leq 1=:\beta_G$.  By Lemma \ref{main-lemma}, we know that
\[\frac{\df_x(X,G)}{1-\lambda}< \frac{\sigma}{\sigma-5+1}< 3.1<\frac{6}{\sqrt[3]{6}}.\]
By Proposition \ref{Prop:Helmke'sInductionCriterion}, we can find a new divisor $G_1$ linearly equivalent to $\lambda_1L$ such that $(X, G_1)$ is log canonical at $x$ and the minimal log canonical center $Z_1=M(G_1)$ is properly contained in $Z$. Moreover,
\[\frac{\df_x(X,G_1)}{1-\lambda_1}< \frac{\df_xG}{1-\lambda}<3.1.\]
If $Z_1$ is a curve, then the induction can be easily proceeded. 

If  $Z_1$ a surface and $\mult_xZ_1\leq 3$, then \[\frac{\df_x(X,G_1)}{1-\lambda_1}<\dfrac{6}{\sqrt[\dim Z_1]{\mult_xZ_1}}.\]

If $Z_1$ is a surface and $\mult_xZ_1=4$, then  $\df_x(X, G_1)\leq \mld_x(Z_1)\leq \frac{1}{2}=:\beta_{G_1}$ by Lemma \ref{mld-rational-surface}. Apply Lemma \ref{secondreductioncriterion}, we see that
\[ \dfrac{\df_x(X, G_1)}{1-\lambda_1}\leq\dfrac{\frac{1}{2}}{\frac{\sigma-5+1}{\sigma}-\frac{\sqrt[3]{6}(1-\frac{1}{2})}{\sigma}}<2.75<\dfrac{6}{\sqrt[\dim Z_1]{\mult_xZ_1}}=3.\]

\case{Assume that  $\mult_xZ=2$.}

In this case, $\df_x(X, G)\leq \beta_G=2$. We have 
\[\frac{\df_x(X, G)}{1-\lambda}\leq \frac{2\sigma}{\sigma-5+2}\leq 4.1<\frac{6}{\sqrt[3]{2}}.\]
We can find a new divisor $G_1$  log canonical at $x$ such that the minimal log canonical center $Z_1=M(G_1)$ is properly contained in $Z$.  Then $Z_1$ is either a smooth curve or a surface with $m_1=\mult_xZ_1\leq 4$. Similar to the previous case, we only need to consider that $\dim Z_1=2$, $m_1\leq 4$. In this case, we know that $\mld_x(X, G_1)\leq \frac{2}{m_1}$. By Lemma \ref{secondreductioncriterion}, we see that for each $m_1\leq 4$,
\[ \dfrac{\df_x(X, G_1)}{1-\lambda_1}\leq\dfrac{\frac{2}{m_1}}{\frac{\sigma-5+2}{\sigma}-\frac{\sqrt[3]{2}(2-\frac{2}{m_1})}{\sigma}}<\dfrac{6}{\sqrt{m_1}}.\]

\case{\textbf{Assume that  $\mult_xZ=3$.}}

In this case, $\df_x(X, G)\leq \beta_G=2$. 
We have  \[\frac{\df_x(X, G)}{1-\lambda}\leq \frac{2\sigma}{\sigma-5+2}\leq 4.1<\frac{6}{\sqrt[3]{3}}.\]
 Then there is a new divisor $G_1$ log canonical at $x$ such that the minimal log canonical center $Z_1=M(G_1)$ is properly contained in $Z$.
 
  If $\df_x(X, G)\leq 1$, then we can apply the same argument as in the case \ref{case:3-fold-m>=4}. 
  
  We may and will assume that $\df_x(X, G)>1$. Let $H$ be a general hyperplane section through $x$. Then  $\df_x(H, G|_H)=\df_x(X, G)-1$ and the minimal log canonical center $Z\cap H$ of $G|_H$ is the same as $Z\cap H$ due to the fact that $\df_x(X, G)>1$.   Therefore, we get  $\mult_x(Z\cap H)=3$ which implies that  $\df_x(H, G|_H)\leq 2/3$ by Lemma \ref{mld-rational-surface}. Hence we may take $\beta_G=\frac{5}{3}$.  If $m_1=\mult_xZ_1\leq 3$, then
\[ \dfrac{\df_x(X, G_1)}{1-\lambda_1}\leq\dfrac{\frac{2}{m_1}}{\frac{\sigma-5+\frac{5}{3}}{\sigma}-\frac{\sqrt[3]{3}(\frac{5}{3}-\frac{2}{m_1})}{\sigma}}<\dfrac{6}{\sqrt[\dim Z_1]{\mult_xZ_1}}.\]
 
 Now we consider the case that $m=\mult_xZ=3$ and $Z_1$ is a surface with multiplicity $m_1=4$.
 We write $S=Z_1$ to remind us that $Z_1$ is a surface.  If a general hyperplane section of $Z$ at $x$ is not an ordinary triple point, i.e. the minimal resolution consists of only one exceptional curve whose self-intersection is  $-3$, then by Lemma \ref{mld-rational-surface} and Proposition \ref{hyperplane} we can take $\beta_G=\frac{8}{5}$ which implies that 
 \[ \dfrac{\df_x(X, G_1)}{1-\lambda_1}\leq\dfrac{\frac{2}{4}}{\frac{\sigma-5+\frac{8}{5}}{\sigma}-\frac{\sqrt[3]{3}(\frac{8}{5}-\frac{2}{4})}{\sigma}}<3.\]

 It only remains the case that a general hyperplane section of $Z$ has a ordinary triple point at $x$ and $\mult_xS=4$. 
  In this cases, the same computation won't give the desired inequality. In fact,
\[ \dfrac{\df_x(X, G_1)}{1-\lambda_1}\leq\dfrac{\frac{2}{4}}{\frac{\sigma-5+\frac{5}{3}}{\sigma}-\frac{\sqrt[3]{3}(\frac{5}{3}-\frac{2}{4})}{\sigma}}<3.05.\]
But we want $\frac{\df_x(X, G_1)}{1-\lambda_1}<3$.  We will show that there exists a certain effective divisor $G'$ such that $\frac{\df_x(X, G')}{1-\lambda'}<3$ and then our induction follows. 

Note that if the divisor $G_1$ already satisfies the inequality $\dfrac{\df_x(X, G_1)}{1-\lambda_1}<3$, we can simply take $G'$ to be $G$.

Assume in the contrary that \begin{equation}\label{bad}
\dfrac{\df_x(X, G_1)}{1-\lambda_1}\geq 3
\end{equation}
for the divisor $G_1=G+cD$ constructed in Lemma \ref{secondreductioncriterion}. 
Recall that $\ord_x(D|_Z)=\df_x(X,G)$ with $D$ linearly equivalent to  $\df_x(X,G)\frac{\sqrt[3]{3}}{\sigma}L$, and $c$ is the log canonical threshold of the triple $(X,G, D)$ at $x$. The inequality \eqref{bad} implies that $\df_x(X,G)$, $\lambda$, $c$ can only vary in a small region. 

\begin{lemma}\label{lemma:lambda-s-c}
Assume that 
\begin{equation*}
\dfrac{\df_x(X, G_1)}{1-\lambda_1}\geq 3
\end{equation*} 
for the general effective $\mathbb{Q}$-divisor $G_1=G+cD$ constructed in Lemma \ref{secondreductioncriterion}.
Then we have 
\begin{enumerate}[ref=\alph*]
\item \label{fail-3-conclusion-1} $1.63\leq \df_x(X,G)\leq \frac{5}{3}$.\\
\item \label{fail-3-conclusion-2} $\lambda \leq \frac{5-\df_x(X, G)}{\sigma}$.\\
\item \label{fail-3-conclusion-3} Let $s=\ord_S(\diff_Z(G))$, then $s\leq 5-3\df_x(X, G)\leq 0.11$.
\end{enumerate}
\end{lemma}

\begin{proof}
If $ \df_x(X, G)<1.63$, we take $\beta_G=1.63$, then  $\frac{\df_x(X, G_1)}{1-\lambda_1}<3$  by Lemma \ref{secondreductioncriterion}, which contradicts with our assumption. So we assume that $1.63\leq \df_x(X, G)\leq \frac{5}{3}$. From the upper bound formula $\df_x(X, G)\leq 5-\lambda\sigma$, we know that $\lambda \leq \frac{5-\df_x(X, G)}{\sigma}$. This proves \eqref{fail-3-conclusion-1} and \eqref{fail-3-conclusion-2}. 

For \eqref{fail-3-conclusion-3}, take a general hyperplane section $H$ in $X$, by Proposition \ref{hyperplane}, we have $\df_x(X, G)=\df_x(H, G|_H)+1$ and $Z\cap H=M(G)\cap H=M(G|_H)$. Note that by definition of different, we know that $\diff_{Z\cap H}(G|_H)\geq (\diff_Z(G))|_H$. Assume that $\diff_Z(G)=sS+T$ where $S\not\subset\supp(T)$.  Then by the definition of deficit, and Proposition \ref{def-up-new} and \ref{def-pairs}, we have
\[\begin{split}\df_x(X, G)&=1+\df_x(H, G|_H) \leq 1+ \df_x(Z\cap H, \diff_{Z\cap H}(G|_H))\\
& \leq 1+\df_x(Z\cap H, (\diff_Z(G))|_H)\\
&\leq 1+ \mld_x(Z\cap H)-\ord_x((\diff_Z(G))|_H)\\&\leq 1+ \mld_x(Z\cap H)-s\cdot\ord_x(S|_H).\end{split}\]
Recall by Lemma \ref{mld-rational-surface}, we may assume that the minimal resolution of the rational surface singularity $(Z\cap H, x)$ consists of only one curve whose self-intersection $-3$. In this case, we know that $\mld_x(Z\cap H)=\frac{2}{3}$ and the index of any Weil divisor on $Z$ is at most $3$ (see \cite[Section 17 and 14]{Lipman1969}). Hence $\ord_x(S|_H)\geq \frac{1}{3}$ and $\df_x(X, G)\leq \frac{5}{3}-\frac{s}{3}$ which implies that $s\leq 5-3\df_x(X, G)$. Note that the inequality $\df_x(X, G)\leq 1.63$ implies that $s\leq 0.11$.
\end{proof}

 Let $\tau=\frac{6(1-\lambda)}{\sqrt[3]{3}}$ and \[L_{\lambda, k}(t)=\{D\mid D \text{~is~} \mb{Q}\text{--Cartier~ on} ~Z ~\text{and ~}kD\in \lsyst{k(1-\lambda)L|_Z\otimes \mf{m}_x^{\ceil{kt}}}\},\] where $k$ is a sufficiently large divisible integer and  $t\geq \df_x(X, G)$. By asymptotic Riemann-Roch theorem, we know that $L_{\lambda,k}(\tau)\neq \emptyset$.

\paragraph{\bf Stable-center property:} For any $t\in [\df_x(X,G), \tau]$ and sufficiently large divisible $k$, the minimal log canonical center of $(Z,\diff_Z(G)+c_tD_{k, t})$ at $x$ is the surface $S$, where $D_{k, t}\in L_{\lambda, k}(t)$ and $c_t$ is the log canonical threshold of the triple $(Z, \diff_Z(G), D_{k,t})$.

\begin{lemma} \label{twocenters}
 Assume that the stable-center property fails, 
then there is a number $\bar{t}\in (\df_x(X, G), \tau)$ and a divisor $\mathbb{R}$--Cariter divisor $D$ linearly equivalent to $(1-\lambda)L|_Z$ with $\ord_xD\geq t$ such that the minimal log canonical center of $(X,G+c_D\tilde{D})$ is a curve or a point, where $\tilde{D}$ is a lifting of $D$.
\end{lemma}

\begin{proof}
We may assume that for any $t\in[\df_x(X, G), \tau]$, $D_{k,t}\in L_{\lambda, k}(t)$ and $k$ sufficiently large, the minimal log canonical center of $(X,G+c_t\tilde{D}_{k,t})$ at $x$ is a surface $S_{(t)}$. 

For a number $t\in [\df_x(X, G), \tau]$ and a sufficiently large $k$, let $D$ be a divisor in $L_{\lambda, k}(t)$ and  $c$ is the log canonical threshold of the triple $(X, G, \tilde{D})$, where $\tilde{D}$ is a lifting of $D$. We claim that $cD=(1-s)S+T$, where $S\not\subset\supp(T_t)$ is a surface in $Z$ and $s=\ord_S(\diff_Z(G))$.

By Lemma \ref{Ambro-diff}, we note that $$(Z, \diff_Z(G+c\tilde{D}))=(Z, \diff_Z(G)+cD),$$ where $Z$ is the minimal log canonical center of $(X, G
)$ at $x$.  Note that $Z$ is also a log canonical center of $(X, G+c\tilde{D})$. 
By \cite[Theorem 1]{Hacon2014}, we know that $(X, G+c\tilde{D})$ is log canonical in a neighborhood of $Z$ if and only if $(Z, \diff_Z(G)+cD)$ is log canonical. Since $S$ is a log canonical center of $(X, G_1)$, then $S$ must be a log canonical center of $(Z, \diff_Z(G)+cD)$. Otherwise, $(Z, \diff_Z(G)+cD)$ is klt at $S$. Consequently, taking an arbitrary small ample divisor $F$ on $X$ containing $S$ but not $Z$, we see that $(Z, \diff_Z(G)+cD+F|_Z)$ is still klt at $S$, but $(X, G+c\tilde{D}+F)$ is no longer log canonical at $S$.  Therefore, $cD=(1-s)S+T$.

Fix a sufficiently large divisible $k_0$ such that for any $k\geq k_0$ the support of the fixed part of $L_{\lambda, k}(\tau)$ is the same as the fixed part of $L_{\lambda, k_0}(\tau)$. Let $D_\tau$ be a general element in $L_{\lambda, k_0}(\tau)$ and let $S_1$, $\ldots$, $S_r$ be the irreducible components of $D_\tau$. We may assume that $S_1$ is the minimal log canonical center of the pair $(X,G+c_\tau \tilde{D}_\tau)$ at $x$. 

Note that $L_{k,\lambda}(t_1)\subseteq L_{k, \lambda}(t_2)$ for any $\df_x(X, G)\leq t_2\leq t_1\leq \tau$ and any sufficiently large divisible $k$. Take a sufficiently large divisible $k\geq k_0$. For each $i=1, \dots, r$ and a general $D_{k, t}\in L_{\lambda, k}(t)$, we see that $\ord_{S_i}D_{k, t}\leq \ord_{S_i}D_{\tau}$ for any $t\leq \tau$.   Hence, the support of the fixed part of $L_{\lambda, k}(t)$ is contained in $\supp(D_{\tau})$.  Denote by $S_{(k, t)}$ is the minimal log canonical center of $(X, G+c_t\tilde{D}_t)$ at $x$. Then $S_{(k, t)}$ is the only log canonical center of $(Z,\diff_Z(G)+c_tD_{k, t})$.  By the generality of $D_{k, t}$, we know that $S_{(k,t)}$ is in the set $\{S_1, \dots, S_r\}$.  We thus obtain that 
$$c_t=\min\{\frac{1-\ord_{S_i}\diff_Z(G)}{\ord_{S_i}(D_{k,t})}\mid i=1, \dots, r\}.$$ 

By the assumption of the lemma, we may assume that $D_{k, t}$ is a divisor such that the minimal log canonical center $S_{(k, t)}$ of $(X, G+c_{t}\tilde{D}_{k,t})$ at $x$ is another surface, say $S_{(k, t)}=S_2$.  We will show there is  an $\mb{R}$--linear combination \[D(u):=uD_\tau+(1-u)D_{t}\in L_{k, \lambda}(\bar{t})\] for some $u\in [0, 1]$ and $\bar{t}\in [t, \tau]$ such that the minimal log canonical center of $(X, G+c(\bar{u})D(\bar{u}))$ at $x$ is not a curve or a point.
 
Consider the following linear functions of $u$
\[L(u; S_i)=\frac{\ord_{S_i}(D(r)|_Z)}{1-\ord_{S_i}\diff_Z(G)}=\frac{\ord_{S_i}D_\tau-\ord_{S_i}(D_t)}{1-\ord_{S_i}\diff_Z(G)}\cdot u+\frac{\ord_{S_i}(D_t)}{{1-\ord_{S_i}\diff_Z(G)}}.\]
Denote by $\mf{U}$ the finite set  $\{u\in [0,1]\mid L(u; S_1)=L(u; S_i), i=1, \dots, r\}.$
 Note that 
\[c_t=\frac{1-\ord_{S_{2}}\diff_Z(G)}{\ord_{S_{e}}(D_t|_Z)}<\frac{1-\ord_{S_i}\diff_Z(G)}{\ord_{S_i}(D_t|_Z)}\]  
for each $i\neq 2$
and 
\[c_\tau=\frac{1-\ord_{S_1}\diff_Z(G)}{\ord_{S_1}(D_\tau|_Z)}< \frac{1-\ord_{S_{i}}\diff_Z(G)}{\ord_{S_{i}}(D_\tau|_Z)}\]
for each $i\neq 1$.  Note that there is a number $u_{(t)}\in [0, 1]$ such that 
\[\frac{1}{L(u_0; S_1)}=\frac{1-\ord_{S_1}\diff_Z(G)}{\ord_{S_1}(D(r)|_Z)}=\frac{1-\ord_{S_{(t)}}\diff_Z(G)}{\ord_{S_{(t)}}(D(r)|_Z)}=\frac{1}{L(u_{(t)}; S_{(t)})}.\]
Then $\mf{U}$ is nonempty. 
Let $\bar{u}$ is the largest number in $\mf{U}$ such that $L(\bar{u}; S_1)=L(\bar{u}; S_i)$ for some $i=1, \dots, r$. Then the pair $(Z,\diff_Z(G)+c(\bar{u})D(\bar{u}))$ has exactly two log canonical centers $S_1$ and $S_i$ at $x$. Consequently, the intersection $S_1\cap S_{i}$ is a lower dimensional log canonical center of $(X,  G+c(\bar{u})\tilde{D}(\bar{u}))$ at $x$.
\end{proof}

If the stable-center property fails, then Lemma \ref{twocenters} shows that there exists a $G'$ such that the minimal log canonical center of $G'$ at $x$ is a curve or a point, and our induction will run through. 
 
If the stable-center property holds, we will show that there exists a divisor $G'$ satisfying the inequality $\frac{\df_x(X, G')}{1-\lambda'}<3$. The divisor $G'$ that we are looking for is a divisor $G+c_{k, \eta}\tilde{D}_{\eta}$, where $\tilde{D}_{k,\eta}$ is 
 a lift of  a divisor $D_{k, \eta}\in L_{\lambda, k}(\eta)$ to $X$,  $c_{k,\eta}$ is the log canonical threshold of the triple $(X, G, \tilde{D}_{k, \eta})$ and $\eta\geq \df_x(X, G)$. 

By Lemma \ref{twocenters}, we may assume that for any $t\in [\df_x(X, G), \tau]$ and $D_{k, t}\in L_{\lambda, k}(t)$, we have $\ord_S(D_{k, t})\geq 1-s$.

Assume $\df_x(X, G)$, $\lambda$, $c$, $s$ are in the region determined by inequalities in Lemma \ref{lemma:lambda-s-c}. By a volume calculation (Proposition \ref{weildiv} and Lemma \ref{lengthdiv}), we will show that there exists an effective $\mb{Q}$--Cartier divisor $D'$ linear equivalent to $(1-\lambda)L|_Z$ with a larger order of vanishing $\eta:=\ord_xD'|_Z$ so that the induction can be proceeded with the new divisor $G'=G+c'D'$ instead of $G_1$. 

Let $\eta$ be the largest number such that there is an effective $\mb{Q}$--Cartier divisor $D'$ linearly equivalent to $(1-\lambda)L|_Z$ with $\ord_xD'=\eta\geq \df_x(X, G)$ and let $c'$ be the minimal number such that $(X, G'=G+c'D')$ is log canonical at $x$ and the minimal log canonical center of $(X, G')$ is $S$. Since $\df_x(X,G')\leq \min\{\frac{1}{2}, \df_x(X, G)-c'\eta\}$, we know that
\begin{equation}\label{final-def}
\begin{split}
\frac{\df_x(X,G')}{1-\lambda'}&\leq\frac{\min\{\frac{1}{2}, \df_x(X, G)-c'\eta\}}{(1-\lambda)(1-c')}\\
&\leq \frac{\eta}{(1-\lambda)(2(\eta-\df_x(X, G))+1)}\\
&\leq \frac{\sigma\eta}{(\sigma-5+\df_x(X, G))(2(\eta-\df_x(X, G))+1)}
=:f(\eta,\df_x(X, G)).
\end{split}
\end{equation}

Take $A=(1-\lambda)L|_Z$, we have shown that $\psi_S(t)\geq 1-s$ for any $t\in [\df_x(G),\tau]$.  On $(Z,\diff_Z(G))$, we have seen that $\eta\geq\tau\geq \df_x(G)\geq 1.63>\frac{4}{3}(1-s)$. On the other hand, if $\eta\geq 4(1-s)\geq 3.56$,  then $f(\eta,\df_x(G))<3$. So we now assume that $\eta\in [\frac{4}{3}(1-s),4(1-s)]$. Apply Lemma \ref{lengthdiv} and Proposition \ref{weildiv}, we have an lower bound for $\vol(\gamma, A):$

\[\begin{split}\vol(\gamma, A)=&A^3-\vol(0,\gamma, A)\\
\geq & A^3- \left(\int_0^{\gamma}9t^2\d t - \int_{\df_x(G)}^{\gamma} h(t,1-s)\d t\right)\\
\geq &  A^3-3\gamma^3+\int_{\df_x(G)}^{\gamma} h(t,1-s)\d t\\
\geq & 
A^3-3\gamma^3+ \int_{\df_x(G)}^{\gamma} 4(1-s)^2-3(t-2(1-s))^2 \d t  \\
=&A^3-3\gamma^3 + 4(1-s)^2(\gamma-\df_x(X, G))\\
&-((\gamma-2(1-s))^3-(\df_x(X, G)-2(1-s))^3)\\
\end{split}
\]

The inequality $\lambda \leq \frac{5-\df_x(X,G)}{\sigma}$ implies that $A^3=(1-\lambda)^3(L|_Z)^3\geq (0.9999+\df_x(X, G))^3$. Let $\eta'$ be the largest real number such that the following inequality of $\gamma$ holds 
\begin{gather*}(0.9999+\df_x(G))^3-3\gamma^3 + 4(1-s)^2(\gamma-\df_x(X,G))\\ 
-((\gamma-2(1-s))^3-(\df_x(X,G)-2(1-s))^3)\geq 0.\end{gather*}
Consider $\eta'=\eta'(s,\df_x(X,G))$ as a function of $\df_x(X,G)$ and $s$. For a fixed $\df_x(X,G)$, when $s$ increases,  $\eta'$ decreases (since the integrand $4(1-s)^2-3(t-2(1-s))^2$ decreases) and $f(\eta', \df_x(X,G))$ increases.  Take $\eta=\eta'-\epsilon$ for sufficiently small $\epsilon$ and $G'$ constructed as before. We know that 
\begin{eqnarray*}
\frac{\df_x(X,G')}{1-\lambda'}&\leq & f(\eta, \df_x(X,G))\\
&\leq & f(\eta'(0.11,\df_x(X,G)), \df_x(X,G))=:g(\df_x(X,G))
\end{eqnarray*}

We can check that $g(\df_x(X,G))$ is an increasing function of $\df_x(X,G)$ for $\df_x(X,G)\in [1.63,\frac{5}{3}]$. Hence $\frac{\df_x(X,G')}{1-\lambda'}\leq g(5/3)\leq 2.98.$

This completes the proof of our main theorem.
\end{proof}

%\begin{appendices}
%\titleformat{\section}{}{\appendixname~\thesection .}{0.5em}{}

\appendix
 
\newpage
\section{Volumes on Weil divisors \label{sec:vol-Weil-div}}%\appendix
In this section, we will calculate the volume of a grade linear system whose fixed part consists of a Weil divisor component.   

We denote by $Z$ a normal projective variety of dimension $n$, $x$ a closed point in $Z$ and $\mf{m}_x$ the maximal idea of $x$ in $Z$. Denote by $m=\mult_xZ$  the multiplicity of $Z$ at $x$.    Let $S$ be a prime Weil divisor on $Z$ with multiplicity $m'=\mult_xS$.  Let $\Delta$ be an effective $\mb{Q}$--divisor on $Z$ such that $(Z, \Delta)$ is a klt pair. For any function $f$ in $\O_Z$, we define the order of vanishing of $f$ at $x$ as 
\begin{equation}\ord_x(f) =\max\{k\mid f\in\mf{m}_x^k\}. \label{ord-function}\end{equation}
Denote by $\IN(f)= f \textup{~modulo~} \m_x^{\ord_x(f)+1}$ the initial form of a function $f\in \O_Z$.

\begin{lemma}\label{integral-tangent-cone}
Let $Z\subset \mb{P}^N$ be a normal variety of dimension 3 with klt singularity at a closed $x\in Z$ whose embedding dimension is $5$. Assume that a general hyperplane section $Z\cap H$ has an ordinary triple point singularity at $x$, i.e. the minimal resolution consists of only one exceptional curve whose self-intersection number is $-3$. Then the order of vanishing function $\ord_x$ on $Z$ defined by \eqref{ord-function} is a discrete valuation. 
\end{lemma}
\begin{proof}
Notice that $Z$ is a minimal multiplicity variety, i.e. $m=e-n+1$. By \cite[Theorem 2]{Sally1977}, the affine tangent cone of $Z$ at $x$, $C_xZ$ is Cohen-Macaulay. In particular, the depth of $C_xZ$ is $3$. Let $H$ be a general hyperplane section of $Z$, that is defined by an element $f\in \frak{m}/\frak{m}^2$. Let $H'$ be the hyperplane of tangent space defined by $\IN(f)$. Since $\IN(f)$ is not a zero divisor, we have $C_x(Z\cap H)=C_xZ\cap H'$ which is irreducible and reduced. The generality of $H$ implies that $C_xZ$ is irreducible and generically reduced. Since $C_xZ$ is Cohen-Macaulay, generically reduced implies reduced. Hence $C_xZ$ is an integral domain. Thus $\ord_x$ is a discrete valuation (see for example \cite[Theorem 6.7.8]{Huneke2006}). \end{proof}

\begin{lemma}\label{lengthCartierdiv} 
Let $D$ be a $\mathbb{Q}$--Cartier divisor on a normal projective variety $Z$ of dimension $n$ with multiplicity $m'=\mult_xD$ and order of vanishing $a=\ord_xD$.  Assume that $\ord_x$ is a discrete valuation on $Z$.  Then for any rational numbers $t, r\geq0$ and sufficiently large divisible $k$,  the function
\[l(t,r):=\length( \dfrac{\mf{m}_x^{kt}}{\mf{m}_x^{kt+1}+\mf{m}_x^{kt}\cap \O_X(-krD)})\]
is a polynomial of $k$ of degree $n-1$ whose leading coefficient is a function $h(t,r)$ satisfying 
\[h(t,r)\geq \frac{m'r}{(n-2)!}(t-ar)^{n-2}\quad \text{~whenever~}\quad t\geq ar.\] 
\end{lemma}
 
\begin{proof}Since the statement is local in nature, we assume that $Z$ is affine.
We first assume that $r=1$. We may and will assume that $t$ is a rational number. Let $k_0$ be the integer such that $\tilde{D}=k_0D$ is Cartier and $\ord_xD=\ord_x\tilde{D}/k_0$, i.e. $k_0$ computes the order of vanishing of $D$.  Since $k$ is sufficiently large divisible, we will assume that $r=\frac{k}{k_0}$, $kt$ and $ka$ are integers.  

Then $\O_Z(-kD)=(f)^{r}$, where $f$ is a defining equation of $k_0D$. Since $\ord_xD=a$, then $(f)\subset \mf{m}_x^{k_0a}$ but $(f)\not\subset\mf{m}_x^{k_0a+1}$.  

For each $0\leq i\leq ks/k_0-1$,  we let $F_i=\O_Z(-i\tilde{D})\cdot (\m_x^{kt-iak_0}\cap \O_Z(-\tilde{D}))$. Then $F_i=O_Z(-i\tilde{D})\cdot (\m_x^{kt-(i+1)ak_0}\cdot \O_Z(-\tilde{D})) =\O_Z(-(i+1)\tilde{D})\cdot \mf{m}_x^{kt-(i+1)ak_0}$.  Denote by 
\[I_i=F_i\cdot(\O_Z/(\mf{m}_x^{kt+1}+\mf{m}_x^{kt}\cap \O_X(-kD)))=\frac{F_i+\mf{m}_x^{kt+1}+\mf{m}_x^{k(t-a)}\cdot \O_X(-r\tilde{D})}{\mf{m}_x^{kt+1}+\mf{m}_x^{k(t-a)}\cdot \O_X(-r\tilde{D})}\] the image of $F_i$ in $\O_Z/(\mf{m}_x^{kt+1}+\mf{m}_x^{k(t-a)}\cdot \O_X(-r\tilde{D}))$.

Set $I_{-1}=\m_x^{kt}/(\mf{m}_x^{kt+1}+\mf{m}_x^{kt}\cap \O_X(-kD))$. We have the following filtration
\[I_{-1}\supset I_0\supset I_1\supset \cdots\supset I_{k/k_0-1}\supset (0).\]
Then \[l(t,1)\geq \sum\limits_{i=0}^{\frac{k}{k_0}-1}\length(I_{i-1}/I_{i}).\]

We note that 
\[\begin{split} I_{i-1}/I_i&=\frac{F_{i-1}+\mf{m}_x^{kt+1}+\mf{m}_x^{k(t-a)}\cdot \O_X(-r\tilde{D})}{F_i+\mf{m}_x^{kt+1}+\mf{m}_x^{k(t-a)}\cdot \O_X(-r\tilde{D})}\\
&=F_{i-1}/{D},
\end{split}
\]
where
\[\begin{split}
D&=F_{i}+F_{i-1}\cap(\mf{m}_x^{kt+1}+\mf{m}_x^{k(t-a)}\cdot \O_X(-r\tilde{D}))\\
&=\mf{m}_x^{kt-(i+1)ak_0}\cdot \O_Z(-(i+1)\tilde{D})+ \mf{m}_x^{kt-iak_0+1}\cdot\O_Z(-i\tilde{D}) +  \mf{m}_x^{k(t-a)}\cdot \O_X(-r\tilde{D})\\
&=\mf{m}_x^{kt-(i+1)ak_0}\cdot \O_Z(-(i+1)\tilde{D}) + \mf{m}_x^{kt-iak_0+1}\cdot\O_Z(-i\tilde{D})
.
\end{split}
\]

Denote by 
\[\bar{\mf{m}}_x=\frac{\mf{m}_x}{\O_Z(-{k_0}S)\cap\mf{m}_x}=\frac{\mf{m}_x}{\O_Z(-{k_0}S)}.\] Since the multiplicity of the scheme $\tilde{D}$  is $k_0\mult_0(D)=k_0m'$. 
%%%%%% The proof of this fact is in Eisenbud's book Exercise 12.11e.
When $i\leq k/k_0-1$, we see that $kt-iak_0\geq k(t-a)+ak_0$ is sufficiently large if $k$ is suffiently large.  Therefore, we have
\[\length(\frac{\bar{\m}_x^{kt-iak_0}}{\bar{\m}_x^{kt-iak_0+1}})\geq \frac{k_0m'}{(n-2)!}(kt-aik_0)^{n-2}.\] Note that 
\[ \begin{split}
\frac{\bar{\m}_x^{kt-iak_0}}{\bar{\m}_x^{kt-iak_0+1}}&=\frac{\m_x^{kt-iak_0}+\O_Z(-\tilde{D})}{\m_x^{kt-iak_0+1}+\O_Z(-\tilde{D})}\\
&=\frac{\m_x^{kt-iak_0}+\m_x^{kt-iak_0+1}+\O_Z(-\tilde{D})}{\m_x^{kt-iak_0+1}+\O_Z(-\tilde{D})}\\
&=\frac{\m_x^{kt-iak_0}}{\m_x^{kt-iak_0+1}+\O_Z(-\tilde{D})\cap \m_x^{kt-iak_0}}\\
&=\frac{\m_x^{kt-iak_0}}{\m_x^{kt-iak_0+1}+\m_x^{kt-(i+1)ak_0}\cdot \O_Z(-\tilde{D})}
\end{split}
\]

For each $i\geq 0$, we consider the map
\[ 
\begin{split}\phi_{i}: \frac{\m_x^{kt-iak_0}}{\m_x^{kt-iak_0+1}+ \m_x^{kt-(i+1)ak_0}\cdot \O_Z(-\tilde{D})}
&\to I_{i-1}/I_{i}\\
{\bar{g}}&\mapsto\bar{gf^{i}},
\end{split}
\]
where $g$ is an element in $\m_x^{kt-iak_0}$, $\bar{g}$ denotes the image of $x$ in $\frac{\bar{\m}_x^{kt-iak_0}}{\bar{\m}_x^{kt-iak_0+1}}$ and $\bar{gf^{i}}$ denotes the image of  $gf^{i}$ in $I_{i-1}/I_{i}$.
The construction of our filtration implies that $\phi_{i}$ is a well-defined morphism of $\kappa(x)$-modules, where $\kappa(x)$ is the residual field of $x$. 

We now claim that $\phi_i$ is an injective morphism for every $i\leq \frac{k}{k_0}-1$. Assume there is an element $g\in\m_x^{kt-iak_0}$ such that 
\[g\not\in \m_x^{kt-iak_0+1}+\m_x^{kt-(i+1)ak_0}\cdot \O_Z(-\tilde{D}),\] but $\bar{gf^{i}}=0$, i.e. 
\[gf^{i}\in D=\mf{m}_x^{kt-(i+1)ak_0}\cdot \O_Z(-(i+1)\tilde{D}) + \mf{m}_x^{kt-iak_0+1}\cdot\O_Z(-i\tilde{D}).\]  But multiplying $f^{-i}$, we will get $g\in \mf{m}_x^{kt-iak_0+1}+\mf{m}_x^{kt-(i+1)ak_0}\cdot \O_Z(-\tilde{D})$.  Therefore $\bar{g}=0$.

Therefore, 
\[\begin{split}l(t, 1)&= \sum_{i=0}^{\frac{k}{k_0}-1}\length(\frac{\bar{\m}_x^{kt-iak_0}}{\bar{\m}_x^{kt-iak_0+1}})\\
&\geq \sum_{i=0}^{\frac{k}{k_0}-1}\frac{k_0 m'}{(n-2)!}(kt-aik_0)^{n-2}\\
&>  \frac{k_0m'}{(n-2)!}\frac{k}{k_0}(kt-a(\frac{k}{k_0}-1)k_0)^{n-2}\\
&= \frac{m'}{(n-2)!}(t-a)k^{n-1}+ o(k).
\end{split}\]

For general $r$, we let $kr=k'$ and $t'=\frac{t}{r}$. Then 
\[\begin{split}\l(t,r)=l(t', 1)&\geq  \frac{m'}{(n-2)!}(t'-a)^{n-2}(k')^{n-1}+ o(k)\\
&= \frac{m'}{(n-2)!}(\frac{t}{r}-a)^{n-2}(kr)^{n-1}+ o(k)\\
&= \frac{m'r}{(n-2)!}(t-ar)^{n-2}k^{n-1}+ o(k).
\end{split}\]
Therefore, $h(t, r)\geq  \frac{m'r}{(n-2)!}(t-ar)^{n-2}$.
\end{proof}

Let $d$ be a positive integer. For any positive integer $\alpha$, there is a unique decreasing sequence of integers $c(d)>c(d-1)>\cdots>c(1)\geq 0$, called the $d$-th Macaulay coefficients of $\alpha$,  such that 
\[\alpha={c(d)\choose d}+{c(d-1)\choose d-1}+\cdots +{c(1)\choose 1}.\]
Denote
\[\alpha_{<d>}={c(d)-1\choose d}+{c(d-1)-1\choose d-1}+\cdots +{c(1)-1\choose 1}.\]
Here, we use the convention that ${i\choose j}=0$ for $i<j$.

Recall Green's hyperplane restriction theorem (see for example \cite[Theorem 􏰇􏰂􏰄􏰂􏰁􏰄4.2.12]{Bruns1993}).
\begin{thm}[Green] \label{thm:Green-hyperplaen-restriction}
Let $R$ be a  homogeneous $k$-algebra, where $k$ is a infinite field, and $n$ be an positive integer. For any general linear form $h\in R_1$, we have 
\[H(R/(h), n)\leq H(R, n)_{<n>},\]
where $H(\cdot, n)$ denotes the Hilbert function of a graded ring.
\end{thm}

\begin{lemma}\label{lemma:Green-Application}
Under assumptions in Lemma \ref{integral-tangent-cone}, we let $S$ be a prime Weil Divisor on $Z$. Then for any integer $d\geq 1$, we have
\[H_k(d)_{<d>}\geq h_k(d),\]
where \[H_k(d):=\dim H^0(Z, \frac{\mf{m}_x^d}{\mf{m}_x^{d+1}+\mf{m}_x^d\cap\O_Z(-kS)}),\]
\[h_k(d):=\dim H^0(H\cap Z, \frac{\mf{m}_x^d}{\mf{m}_x^{d+1}+\mf{m}_x^d\cap\O_Z(-kS\cap H)}).\]
\end{lemma}
\begin{proof}Denote by $(R, \mf{m}_R)$ the local ring of $Z$ at $x$ and $(T, \mf{m}_T)$ the local ring of  $Z\cap H$  at $x$. Let $\mf{p}$ be the defining ideal of $S$ in $R$, $\mf{h}=(h)$ the defining ideal of $H\cap Z$ in $R$, and $\mf{q}:=\mf{p}T$ be the defining ideal of $H\cap S$ in $T$.  Note that $T=R/\mf{h}$ and $\mf{q}=(\mf{p}+\mf{h})/{\mf{h}}$.  Since locally $\mc{O}_{S,x}\cong R/\mf{p}$ is Cohen Macaulay,  then, for a general hyperplane ideal $\mf{h}$, the ring $R/{\mf{h}+\mf{p}}\cong T/\mf{q}$ is also Cohen Macaulay.  Hence $\mf{q}$ defines an effective Weil divisor on $Z\cap H$.
By our assumption, $\spec(T)$ is $\mb{Q}$--factorial, hence $Z(\mf{q})$ is $\mb{Q}$-Cartier. 
Denote by \[gr_{\mf{m}_R}(R)=\bigoplus\limits_{i=0}^{\infty}\mf{m}_R^i/\mf{m}_R^{i+1}\] and 
\[gr_{\mf{m}_T}(T)=\bigoplus\limits_{i=0}^{\infty}\mf{m}_T^i/\mf{m}_T^{i+1}\] the associated graded algebras.  By Lemma \ref{integral-tangent-cone}, we know that $gr_{\mf{m}_R}(R)$ is an integral domain.  Then $gr_{\mf{m}_T}(T)=gr_{\mf{m}_R/\mf{h}}(R/\mf{h})=gr_{\mf{m}_R}(R)/\IN(\mf{h})$, 
%%%%Exercise 5.3 Page 152 Eisenbund.
where $\IN(\mf{h})=(\IN(h))$ means the initial ideal of $\mf{h}$. Denote by $\mf{p}^{[k]}$ and $\mf{q}^{[k]}$ be the $k$-th reflexive powers of $\mf{p}$ and $\mf{q}$.   

Under the above setup, we see that 
\[H_k(d)=\length(\frac{\mf{m}_R^d}{\mf{m}_R^{d+1}+\mf{m}_R^d\cap \mf{p}^{[k]}})=H(gr_{\mf{m}_R}(R)/{\IN(\mf{p}^{[k]})}, d)\]
and 
\[h_k(d)=\length(\frac{\mf{m}_T^d}{\mf{m}_T^{d+1}+\mf{m}_T^d\cap \mf{q}^{[k]}})=H(gr_{\mf{m}_T}(T)/{\IN(\mf{q}^{[k]})}, d).\]

Apply Theorem \ref{thm:Green-hyperplaen-restriction} to $gr_{\mf{m}_R}(R)$, $\IN(h)$ and $\mf{j}$, we know that 
\[H_k(d)_{<d>}\geq \tilde{h}_k(d)=H(gr_{\mf{m}_R}(R)/(\IN(\mf{p}^{[k]})+\IN(h)), d).\] 

We claim that $\tilde{h}_k(d)\geq h_k(d)$. 
To prove the claim, it suffices to show that 
\[\frac{\IN(\mf{p}^{[k]})+(\IN(h))}{(\IN(h))}\to \IN(\mf{q}^{[k]})\]
is injective for any $k$. Since 
\[\frac{\IN(\mf{p}^{[k]})+(\IN(h))}{(\IN(h))}\subset \IN\left(\frac{\mf{p}^{[k]})+(h)}{(h)}\right)\]
and 
\[\frac{\mf{p}^{[k]}+(h)}{(h)}\subset \mf{q}^{[k]}=\left(\frac{\mf{p}+(h)}{(h)}\right)^{[k]},\]
we see that the claim holds.

The last inequality $\frac{\mf{p}^{[k]}+(h)}{(h)}\subset \mf{q}^{[k]}$ follows from the fact that $h$ is general. Let $Z_0$ be the singular locus of $Z$. Since $Z$ is normal, then $\dim Z_0\leq 1$. For a general hyperplane $H$ passing through $x$, it doesn't contain $Z_0$ or $S$. Hence $H\cap Z_0$  consists of only finitely many points. Let $U$ be an open neighborhood of $x$. After shrinking $U$, we may assume that on $U$, $H\cap Z_0=\{x\}$ and the only singular point of $S$ is $x$. Near any point $y\neq x$ in $H\cap S\cap U$, the divisor $S$ is Cartier and locally defined by an equation $s$. Therefore, locally at $y$, we have $\mf{p}^{[k]}=(s^k)$ and $\mf{q}=(\bar{s})$ where $\bar{s}$ the image of $s$ in $T$. This implies that the inequality $\frac{\mf{p}^{[k]}+(h)}{(h)}\subset \mf{q}^{[k]}$ holds on the open set $U\cap S\setminus \{x\}$. Recall that the for reflexive sheaves, the section doesn't change when removing a codimensional $2$ subset. We have $\mf{q}^{[k]}(U\cap S)=\mf{q}^{[k]}(U\cap S\setminus\{x\})$. We thus have $\frac{\mf{p}^{[k]}+(h)}{(h)}\subset \mf{q}^{[k]}$ in $T$.
\end{proof}

\begin{lemma}\label{lengthdiv}
Let $Z$ be a normal variety of dimension 3 and $S$ be a prime Weil divisor on $Z$ with multiplicity $m'=\mult_xS$. Assume that $\ord_x$ is a discrete valuation on $Z$.  Then for any rational numbers $t, r\geq0$ and sufficiently large divisible $k$,  the function
\[l(t,r):=\length( \dfrac{\mf{m}_x^{kt}}{\mf{m}_x^{kt+1}+\mf{m}_x^{kt}\cap \O_X(-krS)})\]
is a polynomial of $k$ of degree $2$ whose leading coefficient is a function $h(t,r)$ satisfying 
\begin{equation}\label{eq:h(s,t)}
h(t,r)\geq \begin{cases}
4tr-8r^2 & t\geq 4r\\
8r^2-(t-4r)^2 & 4r> t\geq 2r\\
4r^2-3(t-2r)^2 & 2r> t\geq \frac{4}{3}r\\
\frac{3}{2}t^2 & \frac{4}{3}r> t\geq 0
\end{cases}.
\end{equation}
\end{lemma}
\begin{proof}
Set $d=kt$ and $k'=kr$.  We first let $r=1$. Applying Lemma \ref{lengthCartierdiv} to  $Z\cap H$ and $S\cap H$,  we see that \[
LT_k(h_{k}(d))\geq \min\{4k, 3d\},
\]
where $LT_k$ means the leading term with respect to $k$.
Note that $LT_k(H_{k}(d))\leq  \frac{3}{2}d^2$. Write
\[H_{k}(d)={c(d)\choose d}+{c(d-1)\choose d-1}+\cdots + {c(1)\choose 1}.
\]
Then $c(d)\leq d+2$, otherwise $\deg LT_k(H_{k}(d))\geq 3$.  Let $x$ be the number of $c(i)$ such that $c(i)= i+2$ and $y$ be the number of $c(i)$ such that $c(i)=i+1$. Since  $LT_k(H_{k}(d))\leq \frac{3}{2}d^2$, then $x\leq 3$. 
By Lemma \ref{lemma:Green-Application},  we know that $H_{k}(d)_{<d>} \geq h_{k}(d)$. Hence $LT_k(H_{k}(d)_{<d>}) \geq \min\{4k, 3d\}$.
In particular, the following inequalities hold
\begin{gather*}
d\geq x+y, \quad\quad LT_k(H_{k}(d)_{<d>})=LT_k(xd+y)\geq \min\{4k, 3d\}\\
LT_k(H_{k}(d))=LT_k(x\frac{d^2}{2}+yd-\frac{y^2}{2})\geq \frac{3}{2}d^2
\end{gather*}

\begin{enumerate}
\item Suppose that $t\geq 4$. We see that $\min\{4k, 3d\}=4k$ and  $d \geq LT_k(y) \geq \max\{4k-xd, 0\}$. Note that $yd-\frac{y^2}{2}$ is increasing for $0\leq y\leq d$.  If $x\geq 1$, then $\max\{4k-xd, 0\}=0$ and \[LT_k(H_{k}(d))\geq \frac{d^2}{2}=\frac{t^2}{2}k^2.\] If $x=0$, then $\max\max\{4k-xd, 0\}=4k$ and  \[LT_k(H_{k}(d))\geq LT_k(4kd-\frac{(4k)^2}{2})=(4t-8)k^2.\]
Since $\frac{t^2}{2}\geq 4t-8$,  then  $h(t, 1)\geq 4t-8$.

\item Suppose that $4\geq t\geq 2$. We see that $\min\{4k, 3d\}=4k$ and  $d \geq LT_k(y) \geq \max\{4k-xd, 0\}$. Then $d\geq 4k-xd$  which implies that $x\geq 1$.  If $x\geq 2$, then $\max\{4k-xd, 0\}=0$ and \[LT_k(H_{k}(d))\geq d^2=t^2k^2.\] If $x=1$, then $\max\{4k-d, 0\}=4k-d$ and 
\[LT_k(H_{k}(d))\geq LT_k(\frac{d^2}{2}+ (4k-d)d-\frac{(4k-d)^2}{2})=(8-(t-4)^2)k^2\]
Since $t^2\geq 8-(t-4)^2$, then $h(t, 1)\geq 8-(t-4)^2$.

\item Suppose that $2\geq t\geq 4/3$. We see that $\min\{4k, 3d\}=4k$ and $x\geq 2$. If $x=3$, then $LT_k(H_{k}(d))=\frac{3}{2}d^2$. If $x=2$, then $LT_k(y)\geq \max\{4k-2d, 0\}=4k-2d$ and 
\[LT_k(H_{k}(d))\geq LT_k(\frac{d^2}{2}+ (4k-2d)d-\frac{(4k-2d)^2}{2})= (4-3(t-2)^2)k^2.\]
Since $\frac{3}{2}t^2\geq 4-3(t-2)^2$, then $h(t, 1)\geq 4-3(t-2)^2$.

\item Suppose that $\frac{4}{3}\geq t\geq 0$. We see that $\min\{4k, 3d\}=3d$ and $x\geq 2$. If $x=3$, then $LT_k(H_{k}(d))=\frac{3}{2}d^2$. If $x=2$, then $LT_k(y)=d$ and  $LT_k(H_{k}(d)_{<d>}=\frac{3}{2}t^2k^2$.
\end{enumerate}

Replacing $k$ by $k'=kr$ and $t$ by $t'=\frac{t}{r}$, we will obtain the lower bounds for $h(t, r)$ as given in \eqref{eq:h(s,t)}.

\end{proof}

\begin{prop}\label{weildiv} Let $A$ be an effective ample $\mb{Q}$--Cartier divisor and $S$ be a prime Weil divisor on $Z$. Assume that the $\ord_x$ is a discrete valuation on $Z$. Write $\Delta=sS+\Delta'$, where $S\not\subset\supp(\Delta')$. If there is a rational number $t_0>0$ such that $\psi_S(t_0)\geq 1-s$, 
then for any $\gamma\geq t_0$ such that ${\vol(\gamma,A)}\geq0$, we have
\[\vol(0, \gamma, A)\leq \int_0^{\gamma}mnt^{n-1}\d t - \int_{t_0}^{\gamma} h(t,1-s)\d t.\]
\end{prop}

\begin{proof}
Let $k$ be a sufficiently divisible integer and $t$ be a rational number. We may assume that $kt$ is an integer, denoted by $j$. Since $F_{k,t}(A)$ is the fixed part of $|\O_Z(kA)\otimes \mathfrak{m}_x^{j}|$, then  \[H^0(Z, \O_Z(kA)\otimes \mathfrak{m}_x^{j})=H^0(Z, \O_Z(kA)\otimes (\O_Z(-F_{k,t}(A))\cap\mathfrak{m}_x^{j}))\] by definition of fixed part. Recall that the coefficient of $S$ in $F_{k,t}(A)$ is greater or equal to $k\cdot\psi_S(t)$. We have 
\[ H^0(Z, \O_Z(kA)\otimes (\O_Z(-k\cdot\psi_S(t)\cdot S)\cap\mathfrak{m}_x^{j}))=H^0(Z, \O_Z(kA)\otimes \mathfrak{m}_x^{j})\] 
Then we have 

\[\begin{split}
& h^0(Z, \O_Z(kA)\otimes\mathfrak{m}_x^j)-h^0(Z, \O_Z(kA)\otimes\mathfrak{m}_x^{j+1})\\
\leq &\dim \im\left( H^0(Z, \O_Z(kA)\otimes\mathfrak{m}_x^j)\to H^0(Z, \O_Z(kA)\otimes \frac{\mathfrak{m}_x^j}{\mathfrak{m}_x^{j+1}}) \right)\\
=&\dim \im\left(H^0(Z, \O_Z(kA)\otimes (\O_Z(-k\cdot\psi(t)\cdot S)\cap\mathfrak{m}_x^{j}))\phantom{\frac{\mathfrak{m}_x^j}{\mathfrak{m}_x^{j+1}}}\right.\\&\hcell{\hspace{\fill}\left. \to  H^0(Z, \O_Z(kA)\otimes \frac{\mathfrak{m}_x^j}{\mathfrak{m}_x^{j+1}})\right)}\\
\leq & h^0(Z, \O_Z(kA)\otimes \frac{\mathfrak{m}_x^j}{\mathfrak{m}_x^{j+1}})-h^0(Z, \O_Z(kA)\otimes \frac{\mathfrak{m}_x^j}{\mathfrak{m}_x^{j+1}+\O_Z(-k\cdot\psi(t)\cdot S)\cap\mathfrak{m}_x^{j}})\\ =&h^0(Z, \frac{\mathfrak{m}_x^j}{\mathfrak{m}_x^{j+1}})-h^0(Z, \frac{\mathfrak{m}_x^j}{\mathfrak{m}_x^{j+1}+\O_Z(-k\cdot\psi(t)\cdot S)\cap\mathfrak{m}_x^{j}}).\\
\end{split}
\]

Recall that $j=kt$ for sufficiently large $k$. Then $h^0(Z, \mathfrak{m}_x^j/\mathfrak{m}_x^{j+1})$ is a polynomial in $k$ with leading term $\frac{m}{(n-1)!}(kt)^{n-1}$. Since $\psi_S(t)\geq \psi_S(t_0)\geq 1-s$ for $t\geq t_0$, then we have
\[h^0(Z, \frac{\mathfrak{m}_x^j}{\mathfrak{m}_x^{j+1}+\O_Z(-k\cdot\psi(t)\cdot S)\cap\mathfrak{m}_x^{j}})\geq h^0(Z, \frac{\mathfrak{m}_x^j}{\mathfrak{m}_x^{j+1}+\O_Z(-k\cdot(1-s)\cdot S)\cap\mathfrak{m}_x^{j}})\]
for $t\geq t_0$.

 By Lemma \ref{lengthdiv}, we know that \[h^0(Z, \frac{\mathfrak{m}_x^j}{\mathfrak{m}_x^{j+1}+\O_Z(-k\cdot(1-s)\cdot Z)\cap\mathfrak{m}_x^{j}})\geq h(t,1-s)k^{n-1}+o(k).\]
 
 Therefore, by taking Riemann sum, we get
\[\begin{split}\vol(0,\gamma, A) &\leq \lim_{k\to\infty}
\dfrac{\sum\limits_{j=0}^{k\gamma}\frac{mk^{n-1}}{(n-1)!}(\frac{j}{k})^{n-1} -
\sum\limits_{j=kt_0}^{k\gamma} h(\frac{j}{k},1-s)k^{n-1} -o(k) }{\frac{k^n}{n!}}\\
& = \int_0^{\gamma}mnt^{n-1}\d t - \int_{t_0}^{\gamma} h(t,1-s)\d t.
\end{split}
\]
\end{proof}

\newpage

\section{Mld of a rational surface singularity \label{sec:mld-def-surfaces}}
\centerline{by Jun Lu}
\begin{center}{\small\itshape Department of Mathematics\\ East China Normal University\\Shanghai, China\\
Email: jlu@math.ecnu.edu.cn
}
\end{center}

\begin{thm}
Let $(S, o)$ be an rational surface singularities with multiplicity $\mult_oS=m$.  Then the minimal log discrepancy of $S$ at $o$ is at most $\frac{2}{m}$.
\end{thm}
\begin{proof}
We may assume that $S$ is affine. 
Let $f: \tilde{S}\to S$ be the minimal resolution, $Z=\sum_{i=1}^ra_iE_i$ be the fundamental cycle, $K=K_{\tilde{S}/S}=\sum_{i=1}^rb_iE_i$ be the relative canonical divisor on $\tilde{S}$, $E=\sum_{i=1}^rE_i$. 

If $m=2$, then it is easy to check that $\mld\leq 1$.

 For $m\geq 3$, we will prove the theorem by contradiction. 
 
 Assume in the contrary that $\mld> \frac{2}{m}$. Then $b_i>\frac{2}{m}-1$ for each $i=1, \dots, r$ which implies that 
\[K+(1-\frac{2}{m})E>0.\]  
Set $$H=K +\left(1-\frac{2}{m}\right)E>0.$$
Since $o$ is a rational singularity and $Z$ is the fundamental cycle which is anti-nef, then we see that $HZ\leq 0$ and $KZ=-2+m$. However, 
\[\begin{split}HZ&=\left(K+\left(1-\frac{2}{m} \right)Z\right)Z-\left(\left(1-\frac{2}{m} \right)(Z-E)\right)Z\\&=- \left(1-\frac{2}{m}\right) (Z-E)Z\geq 0.\end{split}\]
Therefore, $(Z-E)Z=0$ which implies that $EZ=-m$.

We claim that $Z-E=0$. Otherwise, if $Z-E>0$ then $(Z-E)^2<0$ by negative definiteness, then $m+E^2< 0$ equivalently $1+m+E^2\leq 0$. Then $KE=-2-E^2\geq m-1$.  But $m-2=ZK\geq EK\geq m-1$, a contradiction!  This proves that  $Z=E$ under the assumption that $\mld>\frac{2}{m}$.

Now since $E=Z$, $HE=0$, then for any irreducible component $\Gamma$ in $\mathrm{Supp}(H)$ we have $\Gamma E=0$. Hence, 
\[H\Gamma=(K +\left(1-\frac{2}{m}\right)E)\Gamma=K\Gamma=-\Gamma^2-2\geq 0\] which implies
that $H$ is nef, equivalently, $-H$ is anti-nef. Hence $-H\geq 0$ by negativity lemma. But it contradicts to the assumption that $H>0$. 
\end{proof}

%\end{appendices}

%\bibliographystyle{alpha}
%\bibliography{Base-point-freeness}

\begin{thebibliography}{Amb99}

\bibitem[Amb99]{Ambro1999a}
Florin Ambro.
\newblock The adjunction conjecture and its applications.
\newblock {\em arXiv: 9903060}, 1999.

\bibitem[AS95]{Angehrn1995}
Urban Angehrn and Yum~Tong Siu.
\newblock Effective freeness and point separation for adjoint bundles.
\newblock {\em Invent. Math.}, 122(2):291--308, 1995.

\bibitem[BH93]{Bruns1993}
Winfried Bruns and J{\"u}rgen Herzog.
\newblock {\em Cohen-{M}acaulay rings}, volume~39 of {\em Cambridge Studies in
  Advanced Mathematics}.
\newblock Cambridge University Press, Cambridge, 1993.

\bibitem[Dem93]{Demailly1993}
Jean-Pierre Demailly.
\newblock A numerical criterion for very ample line bundles.
\newblock {\em J. Differential Geom.}, 37(2):323--374, 1993.

\bibitem[Ein97]{Ein1997}
Lawrence Ein.
\newblock Multiplier ideals, vanishing theorems and applications.
\newblock In {\em Algebraic geometry---{S}anta {C}ruz 1995}, volume~62 of {\em
  Proc. Sympos. Pure Math.}, pages 203--219. Amer. Math. Soc., Providence, RI,
  1997.

\bibitem[EL93]{Ein1993a}
Lawrence Ein and Robert Lazarsfeld.
\newblock Global generation of pluricanonical and adjoint linear series on
  smooth projective threefolds.
\newblock {\em J. Amer. Math. Soc.}, 6(4):875--903, 1993.

\bibitem[FG12]{Fujino2012}
Osamu Fujino and Yoshinori Gongyo.
\newblock On canonical bundle formulas and subadjunctions.
\newblock {\em Michigan Math. J.}, 61(2):255--264, 2012.

\bibitem[Fuj88]{Fujita1988}
Takao Fujita.
\newblock {\em Contribution to Birational Geometry of Algebraic Varieties: Open
  Problems; the 23rd International Symposium, Division of Mathematics, the
  Taniguchi Foundation; August 22-27, 1988, Katata}.
\newblock 1988.

\bibitem[Fuj93]{Fujita1993}
Takao Fujita.
\newblock {Remarks on Ein-Lazarsfeld criterion of spannedness of adjoint
  bundles of polarized threefolds}.
\newblock {\em arXiv: 9311013}, alg-geom, November 1993.

\bibitem[Hac14]{Hacon2014}
Christopher~D. Hacon.
\newblock On the log canonical inversion of adjunction.
\newblock {\em Proc. Edinb. Math. Soc. (2)}, 57(1):139--143, 2014.

\bibitem[Hei02]{Heier2002}
Gordon Heier.
\newblock Effective freeness of adjoint line bundles.
\newblock {\em Doc. Math.}, 7:31--42 (electronic), 2002.

\bibitem[Hel97]{Helmke1997}
Stefan Helmke.
\newblock On {F}ujita's conjecture.
\newblock {\em Duke Math. J.}, 88(2):201--216, 1997.

\bibitem[Hel99]{Helmke1999}
Stefan Helmke.
\newblock On global generation of adjoint linear systems.
\newblock {\em Math. Ann.}, 313(4):635--652, 1999.

\bibitem[HS06]{Huneke2006}
Craig Huneke and Irena Swanson.
\newblock {\em Integral closure of ideals, rings, and modules}, volume 336 of
  {\em London Mathematical Society Lecture Note Series}.
\newblock Cambridge University Press, Cambridge, 2006.

\bibitem[Kak00]{Kakimi2000}
Nobuyuki Kakimi.
\newblock {Freeness of adjoint linear systems on threefolds with terminal
  Gorenstein singularities or some quotient singularities}.
\newblock {\em The University of Tokyo. Journal of Mathematical Sciences},
  7(3):347--368, 2000.

\bibitem[Kaw97]{Kawamata1997}
Yujiro Kawamata.
\newblock On {F}ujita's freeness conjecture for {$3$}-folds and {$4$}-folds.
\newblock {\em Math. Ann.}, 308(3):491--505, 1997.

\bibitem[Kol93]{Kollar1993}
J{\'a}nos Koll{\'a}r.
\newblock Effective base point freeness.
\newblock {\em Math. Ann.}, 296(4):595--605, 1993.

\bibitem[Kol97]{Kollar1997}
J{\'a}nos Koll{\'a}r.
\newblock Singularities of pairs.
\newblock In {\em Algebraic geometry---{S}anta {C}ruz 1995}, volume~62 of {\em
  Proc. Sympos. Pure Math.}, pages 221--287, Providence, RI, 1997. Amer. Math.
  Soc.

\bibitem[Kol13]{Kollar2013}
J{\'a}nos Koll{\'a}r.
\newblock {\em Singularities of the minimal model program}, volume 200 of {\em
  Cambridge Tracts in Mathematics}.
\newblock Cambridge University Press, Cambridge, 2013.
\newblock With a collaboration of S{\'a}ndor Kov{\'a}cs.

\bibitem[Lee99]{Lee1999}
Seunghun Lee.
\newblock Remarks on the pluricanonical and the adjoint linear series on
  projective threefolds.
\newblock {\em Comm. Algebra}, 27(9):4459--4476, 1999.

\bibitem[Lip69]{Lipman1969}
Joseph Lipman.
\newblock Rational singularities, with applications to algebraic surfaces and
  unique factorization.
\newblock {\em Inst. Hautes \'Etudes Sci. Publ. Math.}, (36):195--279, 1969.

\bibitem[Rei88]{Reider1988}
Igor Reider.
\newblock Vector bundles of rank {$2$} and linear systems on algebraic
  surfaces.
\newblock {\em Ann. of Math. (2)}, 127(2):309--316, 1988.

\bibitem[Sal77]{Sally1977}
Judith~D. Sally.
\newblock On the associated graded ring of a local {C}ohen-{M}acaulay ring.
\newblock {\em J. Math. Kyoto Univ.}, 17(1):19--21, 1977.

\bibitem[YZ14]{Ye2014}
Fei Ye and Zhixian Zhu.
\newblock Global generation of adjoint line bundles on projective $5 $-folds.
\newblock {\em arXiv:1405.0678}, 2014.

\end{thebibliography}

\end{document}